\documentclass[a4paper, 10pt, twoside]{article}

\usepackage[utf8]{inputenc}
\usepackage{exscale, fullpage,url}
\usepackage[centertags]{amsmath}
\usepackage{color}
\usepackage{amssymb}
\usepackage{amsthm}
\usepackage{dsfont}
\usepackage[all]{xy}
\usepackage{mathrsfs}
\usepackage{upgreek}
\usepackage{comment}
\usepackage{tikz-cd}
\usepackage[hidelinks]{hyperref}
\usepackage{graphicx,fp}
\usepackage{mathtools}
\usepackage{paralist}
\usepackage{calc}
\usepackage{diagbox}
\usepackage{enumitem}

\usepackage[capitalise]{cleveref}
\crefname{enumi}{}{}
\usepackage{stmaryrd}
\newif\ifpdf
\pdffalse
\pdftrue

\numberwithin{equation}{subsection}

\usepackage{thmtools}

\newtheorem{thm}{Theorem}[section]

\newcommand{\myendsymbol}{\ensuremath{\diamondsuit}}

\declaretheorem[
  style=definition,
  title=Conjecture,
  qed={$\myendsymbol$},
  sharenumber=thm,
]{conj}

\declaretheorem[
  style=definition,
  title=Lemma,
  qed={},
  sharenumber=thm,
]{lem}

\declaretheorem[
  style=definition,
  title=Proposition,
  qed={},
  sharenumber=thm,
]{prop}

\declaretheorem[
  style=definition,
  title=Corollary,
  qed={},
  sharenumber=thm,
]{cor}

\declaretheorem[
  style=definition,
  title=Remark,
  qed={$\myendsymbol$},
  sharenumber=thm,
]{rmk}

\declaretheorem[
  style=definition,
  title=Convention,
  qed={$\myendsymbol$},
  sharenumber=thm,
]{cnv}

\declaretheorem[
  style=definition,
  title=Notation,
  qed={$\myendsymbol$},
  sharenumber=thm,
]{nota}

\declaretheorem[
  style=definition,
  title=Definition,
  qed={$\myendsymbol$},
  sharenumber=thm,
]{dfn}

\declaretheorem[
  style=definition,
  title=Example,
  qed={$\myendsymbol$},
      sharenumber=thm,
]{exa}

\declaretheorem[
  style=definition,
  title=Remark,
  qed={$\myendsymbol$},
  sharenumber=thm,
]{rem}

\newcommand{\nd}{\noindent}

\newcommand{\dC}{{\mathds C}}
\newcommand{\dQ}{{\mathds Q}}

\newcommand{\dZ}{{\mathds Z}}

\newcommand{\dL}{{\mathbb L}}

\newcommand{\bD}{{\mathbb D}}

\newcommand{\cA}{\mathcal{A}}

\newcommand{\cC}{\mathcal{C}}
\newcommand{\cD}{\mathscr{D}}
\newcommand{\cE}{\mathcal{E}}

\newcommand{\cM}{\mathcal{M}}
\newcommand{\cN}{\mathcal{N}}
\newcommand{\cO}{\mathcal{O}}

\newcommand{\cR}{\mathcal{R}}
\newcommand{\cS}{\mathcal{S}}

\newcommand{\cU}{\mathcal{U}}

\newcommand{\cCE}{\mathcal{C}\!\mathcal{E}}

\newcommand{\fa}{\mathfrak{a}}

\newcommand{\fg}{\mathfrak{g}}

\newcommand{\D}{\displaystyle}

\DeclareMathOperator{\Spec}{\textup{Spec}\,}

\DeclareMathOperator{\Der}{\textup{Der}}
\DeclareMathOperator{\Aut}{\textup{Aut}}

\DeclareMathOperator{\Sym}{\textup{Sym}}
\DeclareMathOperator{\GL}{\textup{GL}}

\DeclareMathOperator{\pr}{\textup{pr}}
\DeclareMathOperator{\act}{\textup{act}}

\DeclareMathOperator{\im}{\textup{im}}
\DeclareMathOperator{\FL}{\textup{FL}}
\DeclareMathOperator{\id}{\textup{id}}

\DeclareMathOperator{\SL}{\textup{SL}}
\DeclareMathOperator{\coker}{\textup{coker}}
\DeclareMathOperator{\roots}{\textup{roots}}

\newcommand{\MHM}{\textup{MHM}}

\newsavebox\foobox

\newcommand{\suchthat}{\;\ifnum\currentgrouptype=16 \middle\fi|\;}

\DeclareMathOperator{\codim}{codim}

\DeclareMathOperator{\trace}{trace}
\DeclareMathOperator{\Tot}{Tot}

\DeclareMathOperator{\Pic}{Pic}

\DeclareMathOperator{\sheafHom}{\mathscr{H}\kern -3pt\textit{om}\kern 1pt}
\DeclareMathOperator{\sheafTor}{\mathscr{T}\kern -3pt\textit{or}\kern 1pt}
\DeclareMathOperator{\sheafExt}{\mathscr{E}\kern -2pt\textit{xt}\kern 1pt}
\DeclareMathOperator{\sheafDer}{\mathscr{D}\kern -1pt\textit{er}\kern 1pt}
\DeclareMathOperator{\differential}{d\!}
\DeclareMathOperator{\ad}{ad}
\DeclareMathOperator{\Ad}{Ad}

\newcommand{\N}{\mathds{N}}

\renewcommand{\P}{\mathds{P}}

\newcommand{\C}{\mathds{C}}
\newcommand{\Q}{\mathds{Q}}
\newcommand{\Z}{\mathds{Z}}

\renewcommand{\O}{\mathcal{O}}
\newcommand{\Ell}{\mathscr{L}}

\let\originalleft\left
\let\originalright\right
\renewcommand{\left}{\mathopen{}\mathclose\bgroup\originalleft}
\renewcommand{\right}{\aftergroup\egroup\originalright}

\newcommand\restr[2]{{#1_{| {#2}}}}

\newcommand\todoText[1]{{\noindent \colorbox{yellow}{\parbox{\minof{\widthof{#1}}{\dimexpr\linewidth-2\fboxsep}}{#1}}}}

\begin{document}
\title{\Large Duality theory of tautological systems}
\author{\small Paul Görlach and Christian Sevenheck}

\renewcommand{\thefootnote}{}
\footnotetext{
\noindent
CS was partially supported by DFG grant SE 1114/6-1.
}
\renewcommand{\thefootnote}{\arabic{footnote}}

\maketitle

\begin{abstract}
  \noindent We discuss the holonomic dual of tautological systems,
  with a view towards applications to linear free divisors and to homogeneous spaces. As a technical tool, we consider a Chevalley--Eilenberg type complex, generalizing Euler--Koszul technology from the GKZ theory, and show equivariance and holonomicity of it.
  \end{abstract}

\tableofcontents

\section{Introduction} \label{sec:intro}

The purpose of this paper is to discuss the duality theory of some differential systems that are naturally attached to group actions on algebraic varieties. More precisely, we are concerned with the so-called \emph{tautological systems}, which were first considered in \cite{Hot98}, discussed from various points of views (especially towards applications to mirror symmetry) in \cite{TautPeriod1, Bloch_HolRankProb, TautPeriod2, HuangLianZhu},
and then studied thoroughly in our previous paper \cite{GRSSW}. There, we were especially interested in Hodge theoretic aspects of tautological systems associated to homogeneous spaces. By the functoriality properties within the category of mixed Hodge modules, it is particularly important to understand the duality theory of such systems. In the present paper, we study the holonomic dual of general tautological systems in detail. Our results can be considered as a generalization of a similar study for the case where the group is an algebraic torus (leading to the well-known GKZ-systems), in that case, it is a result of Walther (see \cite{Walther-Dual}) that under some suitable hypotheses, the dual system is again a GKZ-system defined essentially by the same initial data.
One consequence of our findings is that such a direct expression of the holonomic dual holds essentially only if the dimension of the group coincides with that of the variety it acts on (which is obviously true in the toric case). In general, we give a cohomological expression of the holonomic dual (see \cref{prop:DualityTautCM} and \cref{thm:dualityTautGorenstein} below).
If tautological systems are constructed functorially (as it is the case for those coming from homogeneous spaces, see \cite{GRSSW}), there is a natural morphism between this system and its dual (the direction of which depends on the precise construction by standard functors).

Let us give a more precise overview about the main results of this paper.
Throughout, we work over the field $\C$. To a regular representation $\rho \colon G \to V$ of a connected reductive algebraic group $G$, an orbit closure $\overline Y \subseteq V$ and a Lie algebra homomorphism $\beta \colon \mathfrak g \to \C$, one can associate the Fourier-transformed tautological system $\hat\tau(\rho,\overline Y, \beta)$. This is a $\mathscr D_V$-module with an explicit cyclic presentation given by the $\cD_V$-ideal generated by the vanishing ideal of $\overline Y \subseteq V$ and the vector fields on $V$ induced by the group action of $G$ with a twist given by $\beta$.
The actual tautological system (called $\tau(\rho,\overline Y, \beta)$ in \cite{GRSSW}) is the total Fourier-Laplace transformation of $\hat\tau(\rho,\overline Y, \beta)$, that is, a cyclic $\cD_{V^\vee}$-module on the dual space $V^\vee$. Since the duality functor and the Fourier-Laplace transformation commute up to sign, it is essentially equivalent to describe $\bD \tau(\rho,\overline Y, \beta)$ and
$\bD \hat\tau(\rho,\overline Y, \beta)$. For this reason, we will in this paper consider almost exclusively the latter object.

Our main tool to describe the holonomic dual of $\hat\tau(\rho,\overline Y, \beta)$ is a Chevalley--Eilenberg type complex, namely, we consider $\hat{T}^\bullet(\rho, \overline Y, \beta) = \mathscr D_V \otimes_{\O_V} \O_{\overline Y} \otimes \bigwedge^{-\bullet} \mathfrak g$
with a differential depending on $\beta$.
We comment below in \cref{sec:LieAlg} on how our complex relates to similar constructions in Lie algebra theory. Let us remark that in the toric case, our complex reduces to the Euler-Koszul complex studied in the theory of GKZ-systems (see \cref{rmk:EK-complex} below).

One of our main results can be summarized as follows.
\begin{thm}[see \cref{prop:DualityTautCM} and \cref{thm:dualityTautGorenstein} for more details]\label{thm:MainResult}
The complex $\hat T(\rho, \overline Y, \beta)$ has cohomological amplitude in $\{n-m,\ldots,0\}$ for $n := \dim(\overline Y)$, $m := \dim G$. Moreover, we have
\[\hat\tau(\rho, \overline Y, \beta) \cong H^0 \hat T(\rho, \overline Y, \beta)
\quad\quad\text{ and }\quad\quad
\mathbb D\hat\tau(\rho, \overline Y, \beta) \cong H^{n-m} \hat T(\rho, \overline Y, \tilde \beta),\]
where $\tilde \beta \colon \mathfrak g \to \C$ is a Lie algebra homomorphism potentially different from $\beta$ and the second isomorphism assumes moreover that $\overline Y \subseteq V$ is Gorenstein (e.g.\ a complete intersection).

In particular, if $\dim(G)=\dim(\overline{Y})$, then $\bD\hat\tau(\rho, \overline Y, \beta)$ is again a tautological system for $\rho$, $\overline{Y}$ and $\tilde\beta$.
\end{thm}
A special case where the assumptions of \cref{thm:DualityMain} are satisfied and where we have $\dim\overline{Y}=\dim(G)$ is constructed from so-called linear free divisors. In this case, we use the general duality result to strengthen considerably the main result from \cite{NarvaezSevenheck}, namely, we show that in this case $\hat{\tau}(\rho,\overline{Y},\beta)$ underly a complex mixed Hodge module for all but at most finitely many values of $\beta$ (see \cref{prop:casesLFD} and \cref{cor:AllButFinteLFD}), and that moreover under a stronger assumption on $\beta$, the actual tautological system $\tau(\rho,\overline{Y},\beta)$ associated to a given linear free divisor has an irreducible monodromy representation (\cref{prop:casesLFD}, 3.).

In \cref{sec:MIsNPlusOne} we comment about a few more special cases where under additional assumptions we get sharper duality results. We show in particular that if $\dim(G)=\dim\overline{Y}+1$, then in some cases the dual of $\hat\tau(\rho,\overline{Y},\beta)$ is a again tautological system for $\rho$ and $\overline{Y}$ and a possibly different $\beta$.

One of the main motivations for our work come from a central result of \cite{GRSSW} (namely, Theorem~6.14 in loc.cit.), where we studied tautological systems defined by homogeneous spaces: The  system $\tau(\rho, \overline{Y}, \beta)$ (the same statement holds for the system $\hat\tau(\rho, \overline{Y}, \beta)$) underlies a (in general complex) mixed Hodge module which has a weight filtration of length at most two. Therefore, it is of particular interest to understand the only possible non-trivial weight filtration step. By the functorial construction of $\hat\tau$ (resp. of $\tau$) (see again Theorem~6.14, point 2.\ in loc.cit.), if the weight filtration is non-trivial (i.e., of length two) we have a duality morphism $\tau(\rho, \overline{Y}, \beta)\rightarrow \bD \tau(\rho, \overline{Y}, \beta)$, and the non-trivial weight step is the kernel of this morphism. We postpone the study of this duality morphism to a subsequent paper, however, we do make in \cref{sec:HomSpaces} a few comments and conjectures about it.
It should be noticed that in the case where the group $G$ is an algebraic torus, and when considering GKZ-systems defined by an action of a torus, this duality morphism has been computed explicitly in \cite[Lemma 2.12]{RS12} based on \cite[Proposition 1.15]{Reich2}. This computation is a key step in establishing a mirror symmetry statement for nef complete interse ctions in toric varieties using non-affine Landau-Ginzburg models (see \cite[Theorem 1.10]{RS12}).

\textbf{Acknowledgements:} We would like to thank Thomas Reichelt and Uli Walther for many fruitful discussions about the topic of this paper and Christian Lehn and Patrick Graf for answering our questions about dualizing modules.

\section{Duality for Lie algebroids} \label{sec:LieAlg}

In this section, we exhibit a general duality result in the realm of Lie algebroids, which we will in later sections apply to our specific situation.

\subsection{Chevalley--Eilenberg complexes for Lie algebroids}

We recall for the reader's convenience basic facts about Lie algebroids and their Chevalley--Eilenberg complexes and fix notations. Most of them are easily found in the literature, see \cite{Rinehart, Che99}. We also prove a certain self-duality result on the level of the Chevalley--Eilenberg complex for Lie algebroids, for related statements, see e.g.\ \cite{Huebschmann}.

We recall that a Lie algebroid $(\mathcal E, [\cdot, \cdot], Z)$ on a smooth complex variety $X$ is an $\O_X$-module $\mathcal E$ equipped with a Lie bracket $[\cdot, \cdot] \colon \mathcal E \times \mathcal E \to \mathcal E$ (assumed to be $\C$-bilinear, alternating and satisfying the Jacobi identity), making $\mathcal E$ a sheaf of Lie algebras over $\C$, and together with a map $Z \colon \mathcal E \to \Theta_X$, called the anchor map of $\mathcal E$, which is both a morphism of $\O_X$-modules and a morphism of sheaves of Lie algebras over $\C$. We will mostly work with Lie algebroids $(\mathcal E, [\cdot, \cdot], Z)$ such that the $\O_X$-module $\mathcal E$ is locally free of finite rank. We often speak of a Lie algebroid $\mathcal E$, with the data of the Lie bracket and the anchor map understood implicitly.

Given a Lie algebroid $\mathcal E$, there is a universal enveloping algebra $\mathcal U(\mathcal E)$. We will consider (left or right) $\cU(\cE)$-modules which will always be assumed to be $\O_X$-quasi-coherent. If $\cE$ is locally free of finite type as $\O_X$-module, a PBW-theorem holds for $\cU(\cE)$ (see, e.g., \cite[Theorem 3.1]{Rinehart}), i.e., there is canonical filtration on $\cU(\cE)$ such that its graded algebra is isomorphic to the symmetric algebra (over $\cO_X$) of $\cE$.
In particular, by standard arguments in homological algebra (see, e.g., \cite[Appendix (A.25)]{NarvaezDual} or \cite[Appendix IV, Proposition 4.14]{Bjoerk}), $\cU(\cE)$ has finite global dimension. Henceforth, when speaking of the derived category of $\cU(\cE)$-modules, we always mean the correponding bounded derived category.

For a Lie algebroid $\mathcal E$ that is locally free of rank $m$ as an $\O_X$-module, the line bundle $\omega_{\mathcal E} := \bigwedge_{\O_X}^{m} \mathcal E^\vee$ has a right $\mathcal U(\mathcal E)$-module structure, see, e.g.\ \cite[Appendix (A.20)]{NarvaezDual}.
The tensor product over $\cO_X$ of two left (resp.\ a left and a right) $\cU(\cE)$-modules  is again a left (resp.\ right) $\cU(\cE)$-module as in the theory of $\mathscr D$-modules, see e.g.\ \cite[Proposition 1.2.9 and Proposition 1.2.10]{Hotta} for the case of $\cD$-modules or \cite{NarvaezDual} or \cite[{(2.1) to (2.5)}]{Huebschmann} for more general Lie algebroids. Similarly, $\sheafHom_{\O_X}(\cM, \cN)$ of two left (or two right) $\cU(\cE)$-modules $\cM$, $\cN$ is a left $\cU(\cE)$-module.

In particular, if $\mathcal M$ is a left $\mathcal U(\mathcal E)$-module, then $\omega_\mathcal E \otimes_{\O_X} \mathcal M$ is a right $\mathcal U(\mathcal E)$-module. Conversely, if $\mathcal M'$ is a right $\mathcal U(\mathcal E)$-module, then $\mathcal M' \otimes_{\O_X} \omega_{\mathcal E}^\vee$ is a left $\mathcal U(\mathcal E)$-module.

For any Lie algebroid $\mathcal E$ which is locally free of finite rank $m$ over $\O_X$, with universal enveloping algebra $\mathcal U := \mathcal U(\mathcal E)$, we consider the duality functor
\[\mathbb D_{\mathcal U} :=  R\!\sheafHom_{\mathcal U}(\cdot,\mathcal U)  \otimes_{\O_X} \omega_{\mathcal E}^\vee [m]\]
in the derived category of left $\mathcal U$-modules.
While in this section, we deal with general Lie algebroids, we will later apply the discussion here to one of the following examples:

\begin{exa} \label{ex:lieAlgebraAsLieAlgebroid}
  If $X$ is a point, then a Lie algebroid on $X$ is nothing but a Lie algebra $\mathfrak g$ over $\C$. The universal enveloping algebra of $\mathfrak g$ as a Lie algebroid over a point is the universal enveloping algebra $\mathcal U(\mathfrak g)$ of $\mathfrak g$ as a Lie algebra.
\end{exa}

\begin{exa} \label{ex:DModAsLieAlgebroid}
  For any smooth complex variety $X$, the tangent bundle on $X$ is a Lie algebroid with the anchor map being the identity map. This Lie algebroid is locally free of rank $\dim X$. Its universal enveloping algebra is $\mathcal U(\Theta_X) = \mathscr D_X$. In this case, we denote the duality functor $\mathbb D_{\mathscr D_X}$ simply by $\mathbb D$.
\end{exa}

Notice that, already from these two basic examples, modules over universal enveloping algebras of Lie algebroids can be seen as a uniformizing framework including as special cases both $\mathfrak g$-modules and $\mathscr D_X$-modules.

\begin{exa} \label{ex:AModAsLieAlgebroid}
  If $G$ is a connected linear algebraic group acting on $X$, then $\O_X \otimes \mathfrak g$ is a Lie algebroid on $X$ with bracket given by
  \[[f \otimes \xi, f' \otimes \xi'] := f Z(\xi)(f') \otimes \xi' - f' Z(\xi')(f) \otimes \xi, \qquad f,f' \in \O_X,\; \xi, \xi' \in \mathfrak g\]
  and the anchor map given by $f \otimes \xi \mapsto fZ_X(\xi)$, where $Z_X(\xi)$ is the vector field on $X$ induced by $\xi \in \mathfrak g$ and the group action of $G$ on $X$. Explicitly, $Z_X(\xi)$ is the vector field that associates to a point $y \in X$ the tangent vector $\differential \varphi_y(\xi)$, which is the image of $\xi \in \mathfrak g = T_1 G$ under the differential of the morphism $\varphi_y \colon G \to X$, $g \mapsto g^{-1} \cdot y$. Note that $\O_X \otimes \mathfrak g$ is free of rank $\dim G$ as an $\O_X$-module. We denote the universal enveloping algebra as
  \[\mathcal A_X^G := \mathcal U(\O_X \otimes \mathfrak g) \cong \O_X \otimes \mathcal U(\mathfrak g),\]
  with multiplication given by
  \[(f \otimes \xi) \cdot (f' \otimes \xi') = f f' \otimes \xi \xi' + f Z_X(\xi)(f') \otimes \xi' \qquad f,f' \in \O_X,\; \xi, \xi' \in \mathfrak g.\]
  When no confusion can arise, we may drop the superindex $G$ and simply write $\mathcal A_X$. The anchor map extends to a homomorphism of $\C$-algebras $\mathcal A_X \to \mathscr D_X$ of the universal enveloping algebras.
\end{exa}

\begin{exa} \label{ex:KerAsLieAlgebroid}
  In the previous example, the kernel of the anchor map $Z \colon \O_X \otimes \mathfrak g \to \Theta_X$ is a Lie algebroid on $X$ with the bracket inherited from $\O_X \otimes \mathfrak g$ and the trivial anchor map. If the action of $G$ on $X$ is transitive, then $Z \colon \O_X \otimes \mathfrak g \to \Theta_X$ is a surjection of locally free $\O_X$-modules of rank $\dim G$ and $\dim X$, respectively, hence this Lie algebroid is a locally free $\O_X$-module of rank $\dim G - \dim X$. In general, this Lie algebroid is not of the form of \cref{ex:AModAsLieAlgebroid}, which in particular will mean that the Chevalley--Eilenberg complex considered below does in this case not simply boil down to the classical Chevalley--Eilenberg complex of Lie algebras.
\end{exa}

We next consider the Chevalley--Eilenberg complex of a Lie algebroid.
\begin{dfn} \label{def:SComplex}
Let $\mathcal E$ be a Lie algebroid on a connected smooth complex variety $X$ and assume $\mathcal E$ is locally free of  finite rank $m$ as a $\cO_X$-module. Denote by $\mathcal U := \mathcal U(\mathcal E)$ the corresponding universal enveloping algebra. Moreover, let $\cR$ be any sheaf of associative $\dC$-algebras on $X$.
For any $(\cR,\mathcal U)$-bimodule $\mathcal N$ (this means in particular that the left $\cR$-structure and the right $\cU$-structure on $\cN$ commute), define the Chevalley--Eilenberg complex $\mathcal S^\bullet_{\cR|\cU}(\mathcal N)$
with
\[\mathcal S^{-\ell}_{\cR|\cU}(\mathcal N) := \mathcal N \otimes_{\O_X} \bigwedge_{\O_X}^\ell \mathcal E\]
(note that $\mathcal S^k_{\cR|\cU}(\mathcal N) = 0$ for $k \notin \{-m,\dots,0\}$)
and differential
\begin{align*}
  \delta^{-\ell} \colon \qquad \mathcal N \otimes_{\O_X} \bigwedge_{\O_X}^\ell \mathcal E&\to \mathcal N \otimes_{\O_X} \bigwedge_{\O_X}^{\ell-1} \mathcal E \\
  n \otimes \xi_1 \wedge \dots \wedge \xi_\ell &\mapsto
  \sum_{i=1}^\ell (-1)^{i+1} (n \cdot \xi_i) \otimes \xi_1 \wedge \dots \wedge \widehat{\xi_i} \wedge \dots \wedge \xi_\ell \\
  &\phantom{\mapsto{}}+\sum_{i<j} (-1)^{i+j} n \otimes [\xi_i,\xi_j] \wedge \xi_1 \wedge \dots \wedge \widehat{\xi_i} \wedge \dots \wedge \widehat{\xi_j} \wedge \dots \wedge \xi_\ell. \qedhere
\end{align*}
\end{dfn}
Then $\cS_{\cR|\cU}^\bullet(\cN)$ is a complex of left $\mathcal R$-modules (one checks that $\delta^{-\ell} \circ \delta^{-\ell+1} = 0$). One can prove that $\mathcal S_{\mathcal R|\cU}$ is an exact functor from the abelian category of $(\mathcal R,\mathcal U)$-bimodules to the abelian category of complexes of left $\mathcal R$-modules.

\begin{rem}\label{rem:ConstrRightModule}
The above construction applies in particular if $\cN$ is only a right $\cU$-module, in which case we put $\cR:=\dC:=
\underline{\dC}_X$. Then we recover the construction from
\cite[page 46]{GRSSW}, i.e. $\cS^\bullet_{\dC|\cU}(\cN)$ is a complex of (sheaves of) $\dC$-vector spaces only.

However, for us, the main focus of the above construction is when we start with two locally free Lie algebroids $\cE_1$ and $\cE_2$, with ranks $m_1$ and $m_2$, respectively.
Writing $\mathcal U_1 := \mathcal U(\mathcal E_1)$ and $\mathcal U_2 := \mathcal U(\mathcal E_2)$ for the corresponding universal envelopping algebras, we then get for a $(\mathcal U_1,\mathcal U_2)$-bimodule $\mathcal N$ the Chevalley--Eilenberg complex $\mathcal S^\bullet_{\cU_1|\cU_2}(\mathcal N)$, which is a complex of left $\cU_1$-modules.

Notice that a $(\mathcal U_1,\mathcal U_2)$-bimodule $\mathcal N$ can naturally be seen as a left $\cU_1 \otimes \cU_2^{op}$-module, and again by a PBW-type theorem, the sheaf of rings $\cU_1^{op} \otimes \cU_2$ has finite global homological dimension.

\end{rem}

In the case of \cref{ex:AModAsLieAlgebroid}, when we consider a Lie algebroid of the form $\mathcal E = \O_X \otimes \mathfrak g$ for a $G$-action on $X$, we have the following comparision result relating the above complex with the corresponding object in the theory of Lie algebras.

\begin{lem}\label{lem:ComparisonLieAlgHom}
Let $G$ be a connected linear algebraic group acting on a connected smooth complex variety $X$ and consider the Lie algebroid $\O_X \otimes \mathfrak g$ with universal enveloping algebra $\cU=\cO_X\otimes \cU(\fg)$.
Let $\cN$ be a $(\cR,\cU)$-bimodule. Then we have an isomorphism of complexes of left $\mathcal R$-modules
$$
\left(\cS^{-\bullet}_{\cR|\cU}(\cN),\ \delta\right)
\cong
\left(\cN\otimes \bigwedge^\bullet \fg,\ d \right)=:{\cCE}^\bullet_\fg(\cN)
$$
where the right hand side is the standard Chevalley--Eilenberg complex for \emph{right} modules over Lie algebras (as in \cite[Corollary 7.7.3]{Weibel}). \end{lem}
\begin{proof}
    One immediately checks that the isomorphism
    \[\mathcal S^{-\ell}_{\mathcal R|\mathcal U}(\mathcal N)
    \cong \mathcal N \otimes_{\O_X} \bigwedge_{\O_X}^\ell (\O_X \otimes \mathfrak g)
    \cong \mathcal N \otimes_{\O_X} (\O_X \otimes \bigwedge^\ell \mathfrak g)
    \cong \mathcal N \otimes \bigwedge^\ell \mathfrak g\]
    is compatible with the differential.
\end{proof}
\begin{rem}\label{rem:LieAlgHom}
Notice that in particular that, for $\mathcal E = \O_X \otimes \mathfrak g$, the cohomologies of $\mathcal S^\bullet_{\C|\mathcal U(\mathcal E)}$ correspond to the classical Lie algebra homology in the sense that
$$
H^{-i}\cS^\bullet_{\dC|\cU}(\cN) = H_i(\fg,\cN^\text{left}),
$$
where we consider the right $\mathcal U$-module $\mathcal \cN$ as a \emph{right} module over the Lie algebra $\fg$ via the inclusion $\fg\hookrightarrow \Gamma(X,\cE)$ and $\mathcal \cN^\text{left}$ denotes the corresponding \emph{left} module over $\mathfrak a$ (via $\eta \cdot n := n \cdot (-\eta)$) and where $H_\bullet(\fg,\cN^\text{left})$ denotes the Lie algebra homology of $\fa$ with coefficients in $\cN^\text{left}$ (as in \cite[Def. 7.2.2, Corollary 7.3.6]{Weibel}).\end{rem}

\begin{rem}
In the very simple case where $\cE=\Theta_X$ (so that $\cU(\cE)=\mathscr D_X$), for any right $\mathscr D_X$-module $\cN$, the complex
$\cS^\bullet_{\dC|\mathscr D_X}(\cN)$ is nothing but the well-known \emph{Spencer complex} of $\cN$.
\end{rem}

\textbf{Notation: } Throughout, let $X$ be a connected smooth complex variety. We will for brevity of notation often write $\O$ instead of $\O_X$.

The following basic fact is an analogue of \cite[Lemma~1.2.11]{Hotta}:
\begin{lem} \label{lem:HTTanalogue}
  Let $\cM_1$ and $\cM_2$ be left $\cU$-modules and let $\cN$ be a $(\cR,\cU)$-bimodule for some sheaf of $\C$-algebras $\cR$. Then there are natural isomorphisms
  \[(\cN \otimes_\O \cM_2) \otimes_{\cU} \cM_1 \cong \cN \otimes_{\cU} (\cM_1 \otimes_\cO \cM_2) \cong (\cN \otimes_\cO \cM_1) \otimes_\cU \cM_2\]
  of left $\cR$-modules. An analogous result holds in the derived category of left $\cR$-modules.
\end{lem}

\begin{lem} \label{lem:resolutionOfO}
  Let $\mathcal E$ be a Lie algebroid on $X$, locally free of finite rank as an $\O_X$-module, let $\cU := \cU(\cE)$.
  Then $\mathcal S_{\mathcal U|\mathcal U}^\bullet(\mathcal U)$ is a resolution of $\O_X$ by left $\mathcal U$-modules.
\end{lem}

  This is proven, e.g., in \cite[Lemma 4.1]{Rinehart} (see also \cite[Theorem 2.3.1]{Che99}).

\begin{lem}\label{lem:ComplexSRepresents}
    Let $\mathcal N$ be a $(\mathcal U_1, \mathcal U_2)$-bimodule. The complex $\mathcal S^\bullet_{\cU_1|\cU_2}(\mathcal N)$ represents $\mathcal N \otimes_{\mathcal U_2}^\mathbb{L} \O_X$ in the derived category of left $\mathcal U_1$-modules.
\end{lem}

\begin{proof}
  By \cref{lem:resolutionOfO}, $\cS^\bullet_{\cU_2|\cU_2}(\cU_2)$ is a resolution of $\O_X$ as a left $\mathcal U_2$-module by locally free left $\mathcal U_2$-modules. The claim then follows by observing $\cS_{\cU_1|\cU_2}(\cN) = \cN \otimes_{\cU_2} \cS^\bullet_{\cU_2|\cU_2}(\cU_2)$.
      \end{proof}

\begin{lem} \label{lem:resolution}
  Let $\mathcal E$ be a Lie algebroid on $X$, locally free of finite rank as an $\O_X$-module, let $\cU := \cU(\cE)$.
    Let $\mathcal M$ be a left $\mathcal U$-module and consider $\mathcal U \otimes_{\O_X} \mathcal M$ with the natural $(\mathcal U,\mathcal U)$-bimodule structure given by
  \[\xi \cdot (u \otimes m) = (\xi u) \otimes m, \quad (u \otimes m) \cdot \xi = (u \xi) \otimes m - u \otimes (\xi \cdot m) \quad \text{for } \xi \in \mathcal E,\ u \in \mathcal U,\ m \in \mathcal M.\]
  Then $\mathcal S_{\mathcal U|\mathcal U}^\bullet(\mathcal U \otimes_{\cO_X} \cM)$ is a resolution of $\mathcal M$ by left $\mathcal U$-modules.
\end{lem}

\begin{proof}
  By \cref{lem:ComplexSRepresents}, $S_{\cU|\cU}^\bullet(\cN)$ represents $(\cU \otimes_{\cO_X} \cM) \otimes_\cU^\mathbb{L} \cO_X$. Since $\cU$ is $\cO_X$-locally free (so that obviously $\cU \otimes_{\cO_X} \cM \cong \cU \otimes_{\cO_X}^\mathbb{L} \cM$),  this is by \cref{lem:HTTanalogue} isomorphic to $\cU \otimes_{\cU}^\mathbb{L} (\cO_X \otimes_{\cO_X}^\mathbb{L} \cM) \cong \cM$.
  \end{proof}

  Notice that the statement of the previous lemma can also directly be proven by considering the filtration on $\mathcal N = \mathcal U \otimes_{\cO_X} \mathcal M$ induced by the PBW-filtration on $\mathcal U$ and passing to the graded objects, see \cite[Proof of Lemma 6.9]{GRSSW}. Then we obtain \cref{lem:resolutionOfO} as a special case when $\cM=\cO_X$.

We next observe see a property that can be seen as a certain self-duality of the Chevalley--Eilenberg complex:

\begin{prop} \label{lem:dualityGeneralized}
  Let $\mathcal N$ be a $(\mathcal U_1, \mathcal U_2)$-bimodule and let $\mathcal K$ be a $(\mathcal U_1,\mathcal R)$-bimodule for some sheaf of $\C$-algebras $\mathcal R$.
  There is a natural isomorphism of right $\mathcal R$-modules
   \[\sheafHom_{\mathcal U_1}(S_{\mathcal U_1|\mathcal U_2}^\bullet(\mathcal N),\mathcal K)[m_2] \cong \mathcal S_{\mathcal \C|\mathcal U_2}^\bullet(\omega_{2} \otimes_{\O_X} \sheafHom_{\mathcal U_1}(\mathcal N,\mathcal K)).\]
   In the derived category of right $\mathcal R$-modules, we obtain an isomorphism
   \[R\!\sheafHom_{\mathcal U_1}(\mathcal N \otimes_{\mathcal U_2}^\mathbb{L} \O_X,\mathcal K)[m_2] \cong (\omega_{2} \otimes_{\O_X}^\mathbb{L} R\!\sheafHom_{\mathcal U_1}(\mathcal N,\mathcal K)) \otimes_{\mathcal U_2}^\mathbb{L} \O_X.\]
\end{prop}

\begin{proof}
  For the first claim, in terms of objects of the complex, in cohomological degree $-\ell$, we have
  \begin{align*}
    \sheafHom_{\mathcal U_1}(\mathcal S_{\mathcal U_1|\mathcal U_2}^{-(m_2-\ell)}(\mathcal N),\mathcal K)
    &= \sheafHom_{\mathcal U_1}\bigg(\mathcal N \otimes_{\O_X} \bigwedge_{\O_X}^{m_2-\ell} \mathcal E_2,\mathcal K\bigg) \\
    &\cong \sheafHom_{\O_X}\bigg(\bigwedge_{\O_X}^{m_2-\ell} \mathcal E_2, \sheafHom_{\mathcal U_1}(\mathcal N,\mathcal K)\bigg) \\
    &\cong \bigwedge_{\O_X}^{m_2-\ell} \mathcal E_2^\vee \otimes_{\O_X} \sheafHom_{\mathcal U_1}(\mathcal N,\mathcal K) \\
    &\cong \bigwedge_{\O_X}^{m_2} \mathcal E_2^\vee \otimes \bigwedge_{\O_X}^{\ell} \mathcal E_2 \otimes_{\O_X} \sheafHom_{\mathcal U_1}(\mathcal N,\mathcal K) \\
    &\cong \omega_2 \otimes_{\O_X} \bigwedge_{\O_X}^{\ell} \mathcal E_2 \otimes_{\O_X} \sheafHom_{\mathcal U_1}(\mathcal N,\mathcal K) \\
    &\cong \mathcal S_{\C|\mathcal U_2}^{-\ell}(\omega_{2} \otimes_{\O_X} \sheafHom_{\mathcal U_1}(\mathcal N,\mathcal K)).
  \end{align*}
  It is a tedious, yet straightforward exercise to check that these isomorphisms are compatible with the differentials of the complex. (One needs to be careful to choose the isomorphism
  \[\bigwedge_{\O_X}^{m_2-\ell} \mathcal E_2^\vee \cong \bigwedge_{\O_X}^{m_2} \mathcal E_2^\vee \otimes_{\O_X} \bigwedge_{\O_X}^{\ell} \mathcal E_2\]
  with a suitable sign convention.)

  The second claim follows from the first one by replacing $\mathcal N$ by a finite projective resolution of itself as a $(\mathcal U_1, \mathcal U_2)$-bimodule and applying \cref{lem:ComplexSRepresents}. Note that projective $(\mathcal U_1,\mathcal U_2)$-bimodules are in particular projective as left $\mathcal U_1$-modules.
\end{proof}

\subsection{Duality of \texorpdfstring{$\mathcal U(\mathcal E)$}{U(E)}-modules}

In the following, we consider two Lie algebroids $\mathcal E_1$ and $\mathcal E_2$ on $X$ locally free over $\O$ of finite ranks $m_1$ and $m_2$, respectively. Let $\mathcal U_1$ and $\mathcal U_2$ denote their universal enveloping algebras. Recall that $\omega_i := \bigwedge_{\O}^{m_i} \mathcal E_i^\vee$ has a right $\mathcal U_i$-module structure (given by the negated Lie derivative).

We consider the following double side-changing operation on bimodules: If $\mathcal M$ is a $(\mathcal U_2,\mathcal U_1)$-bimodule, then $\omega_2 \otimes_{\O} \mathcal M \otimes_{\O} \omega_1^\vee$ is a $(\mathcal U_1,\mathcal U_2)$-bimodule. Here, note that the first tensor product uses the $\O$-module structure via $\cO \to \cU_2$ and the left module structure of $\cM$ and the second tensor product uses the $\O$-module structure via $\cO \to \cU_1$ and the right module structure of $\cM$. In general, when working with $(\cU_1,\cU_2)$-bimodules, we always write tensor products in an order that makes implicitly clear which $\cO$-module structure is being used.

It is straightforward to check the following basic property:

\begin{lem} \label{lem:leftRightInterchange}
  Let $\cN$ a $(\cU_2,\cU_1)$-bimodule, let $\cM$ be a left $\cU_1$-module and $\cM'$ a left $\cU_2$-module. There are natural isomorphisms
  \begin{align*}
  \omega_2 \otimes_\O (\cN \otimes_\O \cM) \otimes_\O \omega_1^\vee &\cong \cM \otimes_\O (\omega_2 \otimes_\O \cN \otimes_\O \omega_1^\vee)\\
  \text{and } \qquad \omega_2 \otimes_\O (\cM' \otimes_\O \cN) \otimes_\O \omega_1^\vee &\cong (\omega_2 \otimes_\O \cN \otimes_\O \omega_1^\vee) \otimes_\O \cM' \qquad \phantom{\text{and }}
  \end{align*}
  of $(\cU_1,\cU_2)$-bimodules.\end{lem}

If $\mathcal N$ is a $(\mathcal U_1, \mathcal U_2)$-bimodule, then $\sheafHom_{\mathcal U_1}(\mathcal N, \mathcal U_1)$ naturally inherits a $(\mathcal U_2,\mathcal U_1)$-bimodule structure from the right $\cU_2$-module structure on $\cN$ and the right $\cU_1$-module structure on $\cU_1$. We obtain from \cref{lem:dualityGeneralized} the following duality property:

\begin{prop} \label{lem:dualityProperty}
    Let $\mathcal N$ be a $(\mathcal U_1, \mathcal U_2)$-bimodule.
  There is an isomorphism of complexes of left $\mathcal U_1$-modules:
  \[\sheafHom_{\mathcal U_1}(\mathcal S^{\bullet}_{\cU_1|\cU_2}(\mathcal N),\mathcal U_1)[m_2] \otimes_{\O_X} \omega_{1}^\vee \cong \mathcal S_{\cU_1|\cU_2}^\bullet(\omega_{2} \otimes_{\O_X} \sheafHom_{\mathcal U_1}(\mathcal N,\mathcal U_1) \otimes_{\O_X} \omega_1^\vee).\]
  This gives the following duality result in the derived category of left $\mathcal U_1$-modules:
  \[\mathbb D_{\mathcal U_1}(\mathcal N \otimes_{\mathcal U_2}^{\mathbb L} \O_X) \cong (\omega_2 \otimes_{\O_X} R\!\sheafHom_{\mathcal U_1}(\mathcal N,\mathcal U_1) \otimes_{\O_X} \omega_1^\vee) \otimes_{\mathcal U_2}^{\mathbb L} \O_X[m_1-m_2].\]
\end{prop}

\begin{proof}
  This follows from \cref{lem:dualityGeneralized} for $\mathcal K = \mathcal U_1$ considered as a $(\mathcal U_1,\mathcal U_1)$-bimodule, and passing from right $\mathcal U_1$-modules to left $\mathcal U_1$-modules by taking the tensor product with $\omega_1^\vee$.
\end{proof}

Assume now that there is a homomorphism of Lie algebroids $\cE_2 \to \cE_1$. This induces a $\C$-algebra homomorphism $\mathcal U_2 \to \mathcal U_1$ between their universal enveloping algebras.
In particular, every (left or right) $\cU_1$-module may be viewed as a (left or right) $\cU_2$-module via $\cU_2 \to \cU_1$. On the other hand, for a left $\cU_2$-module $\cM$, we may consider $\cU_1 \otimes_{\cU_2}^\mathbb{L} \cM$ in the derived category of $\cU_1$-modules. This object is represented by the complex $\cS^\bullet_{\cU_1|\cU_2}(\cU_1 \otimes_{\O} \cM)$, because
\[
\cS^\bullet_{\cU_1|\cU_2}(\cU_1 \otimes_\O \cM)
\cong (\cU_1 \otimes_\O^\mathbb{L} \cM) \otimes_{\cU_2}^\mathbb{L} \O
\cong (\cU_1 \otimes_\O^\mathbb{L} \O) \otimes_{\cU_2}^\mathbb{L} \cM
\cong \cU_1 \otimes_{\cU_2}^\mathbb{L} \cM,
\]
where the first isomorphism is due to \cref{lem:ComplexSRepresents} and the second isomorphism is due to \cref{lem:HTTanalogue}.
\begin{rem}\label{rem:LogDiv}
An example of the construction just mentioned is when $D\subseteq X$ is a reduced divisor, and when $\cE_2:=\Der(-\log\,D)$ is the sheaf of logarithmic vector fields. Since our general assumption here is that all Lie algebroids are locally free as $\cO$-modules, we have to restrict to the case where $D$ is a free divisor in the sense of K.~Saito. Let $\cE_1:=\Der_X$, and $\cE_2\rightarrow\cE_1$ be the anchor map of $\cE_2$. Then $\cU_2=\cD_X(-\log\,D)$. Suppose that we are given an integrable logarithmic connection $\cM$, i.e., a left $\cU_2$-module which is also locally free of finite rank as $\cO$-module. Then our construction is exactly the situation studied in, e.g. \cite{CN05, CN09}. In particular, $\cM$ is called admissible in \cite[D\'efinition 1.2.2]{CN05} if $\cU_1\otimes^\dL_{\cU_2}\cM$ is concentrated in degree $0$ and if it is a holonomic $\cU_1$($=\cD_X$)-module.
\end{rem}

In this context, the right $\cU_1$-module $\omega_1$ can also be viewed as a right $\cU_2$-module via $\cU_2 \to \cU_1$, which together with the right $\cU_2$-module structure on $\omega_2$ makes
\[\omega_{1|2} := \sheafHom_\O(\omega_1,\omega_2)\]
a left $\cU_2$-module. This is the relative dualizing module in \cite[Appendix, A.28]{NarvaezDual}.

Through $\cU_2 \to \cU_1$, the algebra $\cU_1$ is in particular a $(\cU_2,\cU_1)$-bimodule or a $(\cU_1,\cU_2)$-bimodule. In fact, these two bimodule structures are related to each other under double side-changing via $\omega_{1|2}$ as follows:

\begin{lem} \label{lem:doubleSideChangingReinterpreted}
  Let $\cE_2 \to \cE_1$ be as above.
    Then there is an isomorphism of $(\cU_1,\cU_2)$-bimodules
  \[\omega_2 \otimes_\O \cU_1 \otimes_\O \omega_1^\vee \cong \cU_1 \otimes_\O \omega_{1|2},\]
  where $\cU_1$ is on the left viewed as a $(\cU_2, \cU_1)$-bimodule and on the right as a $(\cU_1, \cU_2)$-bimodule.
  \end{lem}

\begin{proof}
        We first consider the case $\cE_1=\cE_2$. Note that $\cU_1 \otimes_\O \omega_1^\vee$ is a left $\cU_1 \otimes \cU_1$-module, i.e., it carries two commuting left $\cU_1$-module structures: one via left-multiplication on $\cU_1$ and the other results from the right-to-left transformation $(\cdot) \otimes_\O \omega_1^\vee$ applied to the right $\cU_1$-module structure on $\cU_1$. Note that when we write $\omega_1 \otimes_\O \cU_1 \otimes_\O \omega^\vee$, we mean applying the left-to-right transformation $\omega_1 \otimes_\O (\cdot)$ to the \emph{first} mentioned left $\cU_1$-module structure. However, there is an automorphism of $\cU_1 \otimes_\O \omega_1^\vee$ interchanging the two left $\cU_1$-module structures, given by $P \otimes m \mapsto P \cdot (1 \otimes m)$. (See, e.g., \cite[Corollary~A.12]{NarvaezDual} for the analogous statement for right modules.) Hence, we may as well apply the left-to-right transformation $\omega_1 \otimes_\O (\cdot)$ to the \emph{second} mentioned left $\cU_1$-module structure that arose from a right-to-left transformation. Since the right-to-left and the left-to-right transformations are inverse to each other, this shows $\omega_1 \otimes_\O \cU_1 \otimes_\O \omega_1^\vee \cong \cU_1$.

  For the general case, we observe that $\omega_2 \cong \omega_1 \otimes_\O \omega_{1|2}$ as right $\cU_2$-modules.
      Therefore, as right $\cU_2$-modules, we obtain
  \[     \omega_2 \otimes_\O \cU_1 \otimes_\O \omega_1^\vee
    \cong (\omega_1 \otimes_\O \omega_{1|2}) \otimes_\O (\cU_1 \otimes_\O \omega_1^\vee)
    \cong (\omega_1 \otimes_\O \cU_1 \otimes_\O \omega_1^\vee) \otimes_\O \omega_{1|2}
    \cong \cU_1 \otimes_\O \omega_{1|2}
  \]  and we observe that these isomorphisms are also compatible with the left $\cU_1$-module structure, giving the claimed isomorphism of $(\cU_1,\cU_2)$-bimodules.
\end{proof}

More generally, we get:

\begin{lem} \label{lem:doubleSideChangingModuleVersion}
  Let $\cE_2 \to \cE_1$ be as above and let $\cM$ be a left $\cU_2$-module.
    Then there is an isomorphism of $(\cU_1,\cU_2)$-bimodules
  \[\omega_2 \otimes_\O (\cM \otimes_{\cO} \cU_1) \otimes_\O \omega_1^\vee \cong \cU_1 \otimes_\O \cM \otimes_\cO \omega_{1|2},\]
  where $\cU_1$ is on the left viewed as a $(\cU_2, \cU_1)$-bimodule and on the right as a $(\cU_1, \cU_2)$-bimodule.
\end{lem}

\begin{proof}
  We have
  \[\omega_2 \otimes_\O (\cM \otimes_{\cO} \cU_1) \otimes_\O \omega_1^\vee \cong
  (\omega_2 \otimes_\O \cU_1 \otimes_\O \omega_1^\vee) \otimes_\O \cM
  \cong \cU_1 \otimes_\O \omega_{1|2} \otimes_\O \cM,\]
  where the first isomorphism is from \cref{lem:leftRightInterchange} and the second one from \cref{lem:doubleSideChangingReinterpreted}.
\end{proof}

The following is a consequence of the (self-)duality result for Chevalley-Eilenberg complexes (i.e. \cref{lem:dualityGeneralized}). We would like to point out that the duality result from \cite[Theorem A.32]{NarvaezDual} expresses the duality functor $\mathbb{D}_{\cU_1}(-)$ using
$\mathbb{D}_{\cU_2}(-)$ (and is therefore rather a extension of scalar type property) whereas our result expresses $\mathbb{D}_{\cU_1}(-)$ using $R\!\sheafHom_{\O_X}(-,\O_X)$.
\begin{prop} \label{prop:dualOfScalarExtension}
Let $\mathcal M$ be a left $\mathcal U_2$-module that is coherent as $\O_X$-module. Let $\cE_2 \to \cE_1$ be a Lie algebroid morphism, inducing a homomorphism $\mathcal U_2 \to \mathcal U_1$ that makes $\cU_1$ a $\cU_2$-algebra.
Then
\[\mathbb D_{\mathcal U_1}(\mathcal U_1 \otimes_{\mathcal U_2}^\mathbb{L} \mathcal M)  \cong
\mathcal U_1 \otimes_{\mathcal U_2}^\mathbb{L} (R\!\sheafHom_{\O_X}(\mathcal M,\O_X) \otimes_{\O_X} \omega_{1|2}) [m_1-m_2]\]
in the derived category of left $\mathcal U_1$-modules.
\end{prop}

\begin{proof}
Note that there is a natural morphism
\[R\!\sheafHom_{\O}(\mathcal M, \O) \otimes_{\O} \mathcal U_1 \to R\!\sheafHom_\O(\mathcal M, \mathcal U_1)\]
in the derived category of $(\mathcal U_2,\mathcal U_1)$-bimodules. This is an isomorphism, as can be checked on the level of the underlying objects in the derived category of $\O$-modules---there, this is a consequence of the $\O$-coherence assumption on $\mathcal M$. Now, we have natural isomorphisms
  \begin{align*}
    &\mathbb D_{\mathcal U_1}(\mathcal U_1 \otimes_{\mathcal U_2}^\mathbb{L} \mathcal M) & \\
    &\cong \mathbb D_{\mathcal U_1}((\mathcal U_1 \otimes_\O^\mathbb{L} \mathcal M) \otimes_{\mathcal U_2}^\mathbb{L} \O) &\text{(\cref{lem:HTTanalogue})} \\
    &\cong (\omega_2 \otimes_\O^\mathbb{L} R\!\sheafHom_{\cU_1}(\cU_1 \otimes_{\cO}^\mathbb{L} \mathcal M,\mathcal U_1) \otimes_{\cO}^\mathbb{L} \omega_1^\vee) \otimes_{\mathcal U_2}^\mathbb{L} \O [m_1-m_2] &\text{(\cref{lem:dualityProperty})} \\
    &\cong (\omega_2 \otimes_\O^\mathbb{L} R\!\sheafHom_{\mathcal \O}(\mathcal M,\mathcal U_1) \otimes_{\cO}^\mathbb{L} \omega_1^\vee) \otimes_{\mathcal U_2}^\mathbb{L} \O [m_1-m_2] &\text{(tensor-hom adjunction)} \\
    &\cong (\omega_2 \otimes_\O^\mathbb{L} (R\!\sheafHom_{\mathcal \O}(\mathcal M,\mathcal \O) \otimes_{\O}^\mathbb{L} \mathcal U_1) \otimes_{\cO}^\mathbb{L} \omega_1^\vee) \otimes_{\mathcal U_2}^\mathbb{L} \O [m_1-m_2] &\text{($\O$-coherence of $\mathcal M$)} \\
    &\cong (\mathcal U_1 \otimes_{\O}^\mathbb{L} R\!\sheafHom_{\mathcal \O}(\mathcal M,\O) \otimes_\O^\mathbb{L} \omega_{1|2}) \otimes_{\mathcal U_2}^\mathbb{L} \O [m_1-m_2] &\text{(\cref{lem:doubleSideChangingModuleVersion})} \\
    &\cong \mathcal U_1 \otimes_{\mathcal U_2}^\mathbb{L} (R\!\sheafHom_{\mathcal \O}(\mathcal M,\mathcal \O) \otimes_{\O} \omega_{1|2})[m_1-m_2] &\text{(\cref{lem:HTTanalogue})} & \qedhere
  \end{align*}

\end{proof}

Taking $\mathcal U_1 = \mathcal U_2$, this leads to the following:

\begin{cor}
      Let $\mathcal M$ be an $\O_X$-coherent left $\mathcal U$-module.
  Then
  \[\mathbb{D}_{\mathcal U}(\mathcal M) \cong R\!\sheafHom_{\O_X}(\mathcal M,\O_X)\]
  in the derived category of left $\mathcal U$-modules.
\end{cor}

\begin{proof}
  This follows immediately from \cref{prop:dualOfScalarExtension} and \cref{lem:doubleSideChangingReinterpreted}.
\end{proof}

This corollary is a generalization of \cite[Proposition 3.2.1]{Che99}, in the sense that in loc.\ cit.\ it is assumed that $\cM$ is $\cO_X$-locally free, whereas we only need $\cO_X$-coherence.
Note that for $\mathcal U = \mathscr D_X$, assuming $\cM$ to be $\cO_X$-coherent obviously implies that it is an integrable connection, in which case $\mathbb D(\mathcal M) \cong \sheafHom_{\O_X}(\mathcal M, \O_X)$ (see, e.g., \cite[Example 2.6.10.]{Hotta}). However, for other Lie algebroids, there are cases of $\O$-coherent left $\mathcal U$-modules $\mathcal M$ which are not $\cO$-locally free, and, consequently, for which $\mathbb D_{\mathcal U}(\mathcal M)$ is not a single $\mathcal U$-module in cohomological degree~$0$.

\section{Equivariant aspects}
\label{sec:EquHol}

In this section, we study group actions on smooth varieties which determines a morphism of Lie algebroids as in \cref{ex:AModAsLieAlgebroid}. The results from \cref{sec:LieAlg} can then be made more explicit. Along the way, we consider equivariance properties of the complexes involved and deduce a holonomicity criterion.

\subsection{Holonomicity}

Throughout, we consider the action of a connected linear algebraic group $G$ on a smooth connected variety $X$. We introduce a complex $\cC^\bullet(\cM,\beta)$ of $\mathscr D_X$-modules associated to a $G$-equivariant quasi-coherent $\O_X$-module $\cM$ and a Lie algebra homomorphism $\beta \colon \mathfrak g \to \C$, where $\mathfrak g$ denotes the Lie algebra of $G$.

For this, we consider the algebra $\cA_X = \O_X \otimes U(\mathfrak g)$ as in \cref{ex:AModAsLieAlgebroid} and in \cref{lem:ComparisonLieAlgHom}. We view the $G$-equivariant quasi-coherent $\O_X$-module $\cM$ as a left $\cA_X$-module through the $\mathfrak g$-action obtained by differentiating. Explicitly, the $G$-equivariance of $\cM$ is given by an isomorphism $\varphi \colon \act^* \cM \to \pr_2^* \cM$, where $\act \colon G \times X \to X$ is given by $(g,y) \mapsto g \cdot y$. Any $\xi \in \mathfrak g$ corresponds to a morphism $i_\xi \colon \Spec \C[\varepsilon]/\varepsilon^2 \to G$. Then pulling back $\varphi$ under $i_\xi \times \id_X \colon \Spec \C[\varepsilon]/\varepsilon^2 \times X \to G \times X$ yields an isomorphism
\[(i_\xi \times \id_X)^*\varphi \colon \C[\varepsilon]/\varepsilon^2 \otimes \cM \xrightarrow{\cong} \C[\varepsilon]/\varepsilon^2 \otimes \cM.\]
The image of $1 \otimes m$ under this isomorphism is of the form $1 \otimes m - \varepsilon \otimes m'$. Then $\xi \cdot m := m'$ defines the action of $\mathfrak g$ on $\cM$ which gives the $\cA_X$-module structure on $\cM$.

In the following, let $G$ be a connected linear algebraic group with Lie group $\mathfrak g$, let $X$ be a connected smooth $G$-variety and let $\mathcal M$ be a $G$-equivariant quasi-coherent $\O_X$-module. Let $\beta \colon \mathfrak g \to \C$ be a Lie algebra homomorphism.

\begin{nota}
  For  connected linear algebraic group $G$ acting on a smooth complex variety $X$, a Lie algebra homomorphism $\beta \colon \mathfrak g \to \C$ and any left $\cA_X$-module $\cN$, we denote
  \[\O_X\{\beta\} := \cA_X/\cA_X(\xi - \beta(\xi) \mid \xi \in \mathfrak g) \qquad \text{and} \qquad \cN\{\beta\} := \cN \otimes_{\O_X} \O_X\{\beta\}. \qedhere\]
\end{nota}

\begin{dfn}
  For a $G$-equivariant quasi-coherent $\O_X$-module $\cM$ (which we may view as a left $\cA_X$-module) and Lie algebra homomorphism $\beta \colon \mathfrak g \to \C$, we denote
  \[\cC^\bullet(\cM,\beta):=\cS^\bullet_{\mathscr D_X | \cA_X}(\mathscr D_X \otimes_{\cO_X} \cM\{\beta\}),\]
  where $\mathscr D_X \otimes_{\O_X} \cM\{\beta\}$ is a right $\cA_X$-module by combining the right-multiplication on $\mathscr D_X$ through $\cA_X \to \mathscr D_X$ and the left $\cA_X$-module structure on $\cM\{\beta\}$ resulting from the $G$-equivariance.
  \end{dfn}

Then $\cC^\bullet(\cM,\beta)$ represents $\mathscr D_X \otimes_{\cA_X}^\mathbb{L} \cM\{\beta\}$.

\begin{prop}
  The cohomologies of $\cC^\bullet(\cM,\beta)$ are $\beta$-twisted strongly equivariant $\mathscr D_X$-modules.
\end{prop}

\begin{proof}
  By definition, $\beta$-twisted strong equivariance of   $H^k\cC(\cM,\beta)$
  means for $\act \colon G \times X \to X$, $(g,y) \mapsto g \cdot y$ the existence of an isomorphism of $\mathscr D_{G \times X}$-modules
  \[\act^+ H^k \cC^\bullet(\cM,\beta) \xrightarrow{\cong } \O_G^\beta \boxtimes H^k\cC^\bullet(\cM,\beta)\]
  satisfying a cocycle condition. By exactness of $\act^+ (\cdot)$ and $\O_G^\beta \boxtimes (\cdot)$, it suffices to construct an isomorphism $\act^+ \cC^\bullet(\cM,\beta) \xrightarrow{\cong} \O_G^\beta \boxtimes \cC^\bullet(\cM,\beta)$ in the derived category $D_{qc}^b(\mathscr D_{G \times X})$.

  For a Lie algebra homomorphism $\alpha \colon \mathfrak g \to \C$, the complex of $\mathscr D_G$-modules $\cC^\bullet(\O_G,\alpha)$ is a resolution of $\O_G^\alpha$, since $Z_G \colon \O_G \otimes \mathfrak g \to \Theta_G$ and thus $\tilde Z_G \colon \cA_G \to \mathscr D_G$ is an isomorphism. Hence, in $D_{qc}^b(\mathscr D_X)$, we have
  \[\O_G^\alpha \boxtimes \cC^\bullet(\cM,\beta) \cong \Tot\big(\cC^\bullet(\O_G,\alpha) \boxtimes \cC^\bullet(\cM,\beta)\big) = \cC^\bullet(\O_G \boxtimes \cM,\alpha\oplus \beta),\]
    where $\cC^\bullet(\O_G,\alpha)$ and $\cC^\bullet(\cM,\beta)$ are based on the $G$-action on $G$ (by left-multiplication) and on $X$, respectively, and $\cC^\bullet(\O_G \boxtimes \cM,\alpha\oplus \beta)$ is based on to the component-wise $G \times G$-action on $G \times X$.
  Notice that the last equality follows from the definition of $\cC^\bullet$ using that the Chevalley--Eilenberg complexes on a product $X_1 \times X_2$ with respect to the Lie algebroid $\cE = \pr_1^*\cE_1 \oplus \pr_2^* \cE_2$ satisfy
  \[\cS^\bullet_{\cR | \cU(\cE)}(\cN_1 \boxtimes \cN_2)
  = \Tot^\bullet(\cS_{\cR | \cU(\cE_1)}(\cN_1) \boxtimes \cS_{\cR | \cU(\cE_2)}(\cN_2))\]
  since
  \[\bigwedge_\cO^k (\cN_1 \boxtimes \cN_2) \otimes_\O \cE = \bigoplus_{p+q=k} (\cN_1 \otimes_\cO \bigwedge_\cO^p \cE_1) \boxtimes (\cN_2 \otimes_\cO \bigwedge_\cO^q \cE_2)\]
  and by observing compatibility of the differential.

  For
  \[\psi \colon G \times X \xrightarrow{\cong} G \times X, \qquad (g,y) \mapsto (g,g \cdot y),\]
  we have $\act = \pr_2 \circ \psi$. Hence:
  \begin{align*}
  &\act^+ \cC^\bullet(\cM,\beta) = \psi^+ (\O_G \boxtimes \cC^\bullet(\cM,\beta)) \cong \psi^+ \cC^\bullet(\O_G \boxtimes \cM, 0 \oplus \beta) \\
  &= \psi^+ \Big(\mathscr D_{G \times X} \otimes_{\O_{G\times X}} (\O_G \boxtimes \cM)\{0 \oplus \beta\} \otimes_{\O_{G \times X}} {\textstyle \bigwedge_{\O_{G \times X}}^{-\bullet}}\big((\O_{G \times X} \otimes \mathfrak g) \oplus (\O_{G \times X} \otimes \mathfrak g)\big)\Big) \\
  &= \psi^+ \mathscr D_{G \times X} \otimes_{\O_{G\times X}} \act^* \cM \otimes_{\O_{G \times X}} \psi^*(\O_{G \times X}\{0 \oplus \beta\}) \otimes_{\O_{G \times X}} {\textstyle \bigwedge_{\O_{G \times X}}^{-\bullet}}\psi^*\big((\O_{G \times X} \otimes \mathfrak g) \oplus (\O_{G \times X} \otimes \mathfrak g)\big)\Big)
        \end{align*}

  Note that $\psi$ induces an isomorphism $\psi^\sharp \colon \psi^+ \mathscr D_{G \times X} \xrightarrow{\cong} \mathscr D_{G \times X}$ explicitly given by $\psi^* \O_{G \times X} \xrightarrow{\cong} \O_{G \times X}$ (composing regular functions with $\psi$) and the inverse of $\differential \psi \colon \Theta_{G\times X} \xrightarrow{\cong} \psi^* \Theta_{G \times X}$ (pushing forward vector fields). Moreover, since $\cM$ is $G$-equivariant, there is an isomorphism $\varphi \colon \act^* \cM \to \O_G \boxtimes \cM$ of $\O_{G \times X}$-modules.

  The adjoint action of $G$ on $\mathfrak g$ defines an isomorphism $\Ad \colon \O_G \otimes \mathfrak g \xrightarrow{\cong} \O_G \otimes \mathfrak g$, given on $U \subseteq G$ by understanding local sections $\in \Gamma(U,\O_G \otimes \mathfrak g)$ as morphisms $U \to \mathfrak g$ and mapping $\xi \colon U \to \mathfrak g$ to $\Ad(\xi)\colon U \to  \mathfrak g$, $g \mapsto \Ad(g)(\xi(g))$.
  Under the isomorphism $\psi$, we get on the level of vector fields:
  \[\differential\psi^{-1}(Z_{G \times X}(\xi_1,\xi_2)) = Z_G(\xi_1) + Z_X(\Ad^{-1}(\xi_1-\xi_2))\]
  for all $\xi \in \O_{G \times X} \otimes \mathfrak g$.
  This lifts to a commutative diagram
  \[\begin{tikzcd}
      \psi^*\big((\O_{G \times X} \otimes \mathfrak g) \oplus (\O_{G \times X} \otimes \mathfrak g)\big) \ar{r}{\psi^* Z_{G\times X}} \ar{d}{\chi} & \psi^*\Theta_{G \times X} \ar{d}{\differential\psi^{-1}} \\
      (\O_{G \times X} \otimes \mathfrak g) \oplus (\O_{G \times X} \otimes \mathfrak g) \ar{r}{Z_{G\times X}} & \Theta_{G \times X},
  \end{tikzcd}\]
  where $\chi$ is the isomorphism of Lie algebroids given by $\chi(\xi_1,\xi_2) := (\xi_1,\Ad^{-1}(\xi_1-\xi_2))$.

  Combining these isomorphisms, we get an isomorphism of complexes
  \[\psi^+ \cC^\bullet(\O_G \boxtimes \cM, 0 \oplus \beta) \xrightarrow{\cong} \cC^\bullet(\O_G \boxtimes \cM, \beta \oplus \beta)\]
  which by the above induces an isomorphism $\act^+ \cC^\bullet(\cM,\beta) = \psi^+(\O_G \boxtimes \cC^\bullet(\cM,\beta)) \xrightarrow{\cong} \O_G^\beta \boxtimes \cC^\bullet(\cM,\beta)$ in the derived category $D_{qc}^b(\mathscr D_{G \times X})$. One checks the cocycle condition.
\end{proof}

\begin{cor} \label{cor:holonomy}
  For an equivariant $\O_X$-module $\cM$ supported on finitely many orbits and $\beta \colon \mathfrak g \to \C$ a Lie algebra homomorphism, the cohomologies of $\cC^\bullet(\cM,\beta)$ are holonomic, i.e., $\mathscr D_X \otimes_{\cA_X}^\mathbb{L} \cM\{\beta\} \in D^b_h(\mathscr D_X)$.
\end{cor}

\begin{proof}
  (Twisted) strongly equivariant $\mathscr D$-modules supported on finitely many orbits are holonomic, see e.g.~\cite[§5]{Hot98} or \cite[Theorem 11.6.1]{Hotta} for a more modern account.
\end{proof}

\begin{cor}\label{cor:CohomAmplitude}
    Let $\cM$ be a $G$-equivariant coherent $\O_X$-module supported on finitely many orbits and let $\beta \colon \mathfrak g \to \C$ be a Lie algebra homomorphism. Then, for $k := \max\{i \mid \sheafExt^i_{\O_X}(\mathcal M, \O_X) \neq 0\}$ (if $X$ is affine,  this is simply the projective dimension of $\cM$ as $\cO_X$-module), we have
  \[H^i \cC^\bullet(\cM,\beta) = 0 \qquad \text{for } i \notin [\dim X - \dim G - k,0].\]
  \end{cor}

\begin{proof}
  Recall that the complex $\cC^\bullet(\cM,\beta)$ represents $\mathscr D_X \otimes_{\cA_X}^\mathbb{L} \cM\{\beta\}$. By \cref{prop:dualOfScalarExtension}, we have
  \[\mathbb D(\mathscr D_X \otimes_{\cA_X}^\mathbb{L} \mathcal M\{\beta\})  \cong
\mathscr D_X \otimes_{\mathcal \cA_X}^\mathbb{L} (R\!\sheafHom_{\O_X}(\mathcal M,\O_X) \otimes_{\O_X} \omega_{1|2}) [\dim X- \dim G],\]
      By the definition of $k$
    and right-exactness of $\mathscr D_X \otimes_{\cA_X} (\cdot)$, the right-hand side has no non-zero cohomology in degree larger than $k + \dim G - \dim X$. By \cref{cor:holonomy}, the cohomologies of $\cC^\bullet(\cM,\beta)$ are holonomic, which implies $H^{-i} \mathbb{D}(\cC^\bullet(\cM,\beta)) = \mathbb{D} H^i(\cC^\bullet(\cM,\beta))$. Therefore, $\cC^\bullet(\cM,\beta)\cong\mathscr D_X \otimes_{\cA_X}^\mathbb{L} \cM\{\beta\}$ has no cohomology in degree smaller than $\dim X - \dim G - k$.
      \end{proof}

\begin{rmk}\label{rmk:EK-complex}
  The case of a finite-dimensional regular representation $X=\C^n$ of a torus $G = (\C^*)^d$ corresponds to the study of GKZ-systems. In this case, $G$-equivariant coherent $\O_X$-modules supported on finitely many orbits are exactly toric modules as defined in \cite{MMW}. The complexes $\cC^\bullet(\cM,\beta)$ agree with the Euler--Koszul complexes studied in loc.cit.
\end{rmk}

\subsection{Duality theory}

Our aim is to use \cref{prop:dualOfScalarExtension} to describe the dual of $\cC^\bullet(\cM, \beta)$. For this, we need a little preparatory work to describe the relative dualizing module $\omega_{\mathscr D_V | \mathcal A_V}$ associated to the homomorphism of Lie algebroids $\O_X \otimes \mathfrak g \to \Theta_X$.

The right $\cU(\cE)$-module $\omega_\mathcal E$ for a locally free Lie algebroid $\mathcal E$ of finite rank is in the case $\cE = \Theta_X$, $\cU(\cE)=\mathscr D_X$ (for some smooth variety $X$) the canonical sheaf $\omega_X$ with a right $\mathscr D_V$-action via the Lie derivative. In the case $\cE = \O_X \otimes \mathfrak g$, $\cU(\cE) = \mathcal A_X$,
we denote $\omega_{\O_X \otimes \mathfrak g}$ by $\alpha_X$, a right $\mathcal A_X$-module whose underlying $\cO_X$-module is isomorphic to $\cO_X$.

\begin{lem} \label{lem:alphaDescription}
  We have an isomorphism
$$
\omega_{\mathscr D_X|\cA_X} \cong \omega_X^\vee\{-\trace \circ \ad\}
$$
as a left $\cA_X$-module, where on the right hand side $\omega_X^\vee$ is a left $\cA_X$-module through its structure as a $G$-equivariant line bundle.
\end{lem}

\begin{proof}
 The right $\mathcal A_X$-module $\omega_X$ (with the right $\mathcal A_X$-module structure inherited from the right $\mathscr D_X$-module structure via $\mathcal A_X \to \mathscr D_X$) can be understood as $\alpha_X \otimes_{\mathcal O_X}\omega_X\{\delta\}$ for $\delta := \trace \circ \ad \colon \mathfrak g \to \C$, where now $\omega_X$ is viewed as a left $\mathcal A_X$-module via its structure as an equivariant line bundle on $X$. This can be seen from writing out in detail the right $\mathcal A_X$-action, see \cite[Lemma~4.27]{GRSSW}) for details.

  We have a canonical isomorphism of left $\mathcal A_X$-modules
  \begin{align*}
  \omega_{\mathscr D_X|\mathcal A_X} &= \sheafHom_{\O_X}(\omega_X,\alpha_X) \cong \sheafHom_{\O_X}(\alpha_X \otimes_{\O_X} \omega_X \{\delta\}, \alpha_X) \\
  &\cong \sheafHom_{\O_X}(\omega_X\{\delta\},\sheafHom_{\O_X}(\alpha_X,\alpha_X)) \cong \sheafHom_{\O_X}(\omega_X\{\delta\},\O_X) \cong \omega_X^\vee\{-\delta\}. \qedhere
  \end{align*}
\end{proof}

\begin{cnv} \label{conv:OmegaM} In the following, we consider a $G$-equivariant coherent $\cO_X$-module $\cM$ that is Cohen--Macaulay, i.e., \[\sheafExt_{\O_X}^i(\cM,\omega_X) = 0 \text{ for all }i \neq  \codim \cM
\]
as $\cO_X$-modules. We write $\omega_{\cM} := \sheafExt_{\O_X}^{\codim \cM}(\cM,\omega_X)$ for its dualizing module, which is again a $G$-equivariant coherent $\cO_X$-module (this follows, e.g., from \cite[Proposition 5.1.26]{ChrissGinzburg}).
\end{cnv}

\begin{thm}\label{thm:DualityMain}
  In $D^b_{qc}(\mathscr D_X)$, for $\cM$ Cohen-Macaulay, we have
  \[\mathbb D \, \cC^\bullet(\cM,\beta) \cong \cC^\bullet(\omega_M \otimes_\O (\omega_X^\vee)^{\otimes 2},-\trace\circ \ad-\beta)[\dim \cM-\dim G],\]
  where, as usual, $\dim(\cM)$ means $\dim(\textup{supp}(\cM))$.
\end{thm}

\begin{proof}
     Denote $m := \dim G$ and $n := \dim \cM$. We use \cref{prop:dualOfScalarExtension} to obtain:
    \begin{align*}
      &\mathbb D_{\mathscr D_X} \, \cC^\bullet(\cM, \beta) \\
      &= \mathbb D_{\mathscr D_X}(\mathscr D_X \otimes_{\cA_X}^\mathbb{L} \cM\{\beta\}) \\
      &= \mathscr D_X \otimes_{\cA_X}^\mathbb{L} (R\!\sheafHom_{\cO_X}(\cM\{\beta\},\cO_X) \otimes_{\cO_X} \omega_{\mathscr D_X|\mathcal A_X})[\dim X-m] \hspace{9.6em} \text{(\cref{prop:dualOfScalarExtension})} \\
      &= \mathscr D_X \otimes_{\cA_X}^\mathbb{L} (\sheafExt_{\cO_X}^{\codim \cM}(\cM\{\beta\},\cO_X)[-\codim \cM] \otimes_{\cO_X} \omega_{\mathscr D_X|\mathcal A_X})[\dim X-m] \quad (\cM \text{ is Cohen--Macaulay)}\\
      &= \mathscr D_X \otimes_{\cA_X}^\mathbb{L} (\sheafExt_{\cO_X}^{\codim \cM}(\cM\{\beta\},\cO_X) \otimes_{\cO_X} \omega_{\mathscr D_X|\mathcal A_X})[n-m] \hspace{6.25em} \text{($\codim \cM = \dim X-\dim \cM$)}\\
      &= \mathscr D_X \otimes_{\cA_X}^\mathbb{L} (\sheafExt_{\cO_X}^{\codim \cM}(\cM,\cO_X)\{-\beta\} \otimes_{\cO_X} \omega_{\mathscr D_X|\mathcal A_X})[n-m]  & \\
      &= \mathscr D_X \otimes_{\cA_X}^\mathbb{L} (\sheafExt_{\cO_X}^{\codim \cM}(\cM,\omega_X) \otimes_{\cO_X} \omega_X^\vee\{-\beta\} \otimes_{\cO_X} \omega_{\mathscr D_X|\mathcal A_X})[n-m]  & \\
      &= \mathscr D_X \otimes_{\cA_X}^\mathbb{L} (\omega_\cM \otimes_{\cO_X} \omega_X^\vee\{-\beta\} \otimes_{\cO_X} \omega_X^\vee \{-\trace \circ \ad\})[n-m] \hspace{2.5em}  \text{(definition of $\omega_{\mathscr D_X| \mathcal A_X}$, \cref{lem:alphaDescription})} \\
      &= \cC^\bullet(\omega_\cM \otimes_{\cO_X} (\omega_X^\vee)^{\otimes 2}, -\trace \circ \ad - \beta)[n-m]. \qedhere
    \end{align*}
                                            \end{proof}

\begin{exa}
  Consider the case where $\cM:=\omega_X^{\otimes k}$, so that
  $\omega_M=\omega_X^{\otimes(1-k)}$.
  Then
  \[\mathbb{D}\,\cC^\bullet(\omega_X^{\otimes k},\beta) = \cC^\bullet(\omega_X^{\otimes (-k-1)},-\trace \circ \ad - \beta)[\dim(X)-\dim(G)].\]
  In particular, for $k = 0$ and $k = -1$, we obtain
  \begin{align*}
  \mathbb{D}\,\cC^\bullet(\O_X,\beta) &= \cC^\bullet(\omega_X^{\vee},-\trace \circ \ad - \beta)[\dim(X)-\dim(G)] \\
  \mathbb{D}\,\cC^\bullet(\omega_X^\vee,\beta) &= \cC^\bullet(\O_X,-\trace \circ \ad - \beta)[\dim(X)-\dim(G)]
  \end{align*}
    If the group is unimodular, so $\trace \circ \ad = 0$, and if additionally, the action is transitive, then we get in particular for $\beta = 0$ (by also taking into account that $H^i$ and $\bD$ commute due to \cref{cor:holonomy}) that
  \begin{equation}
      \label{eq:DualityOrbit}
  \mathbb{D}\,H^{n-m}\cC^\bullet(\omega_X^\vee,0) = H^0\cC^\bullet(\O_X,0) = \O_X.
  \end{equation}
\end{exa}

\section{Tautological systems} \label{sec:Duality}

To a regular representation $\rho \colon G \to V$ of a connected linear algebraic group $G$, the closure $\overline Y \subseteq V$ of a $G$-orbit $Y$ and a Lie algebra homomorphism $\beta \colon \mathfrak g \to \C$, one associates the following cyclic left $\mathscr D_V$-module, the \emph{Fourier-transformed tautological system}:
\[\hat\tau(\rho,\overline Y,\beta):= \mathscr D_V/(\mathscr D_V \mathcal I_{\overline Y} + \mathscr D_V (Z_V(\xi)-\beta'(\xi) \mid \xi \in \mathfrak g)),\]
where $\mathcal I_{\overline Y} \subseteq \O_{\overline Y}$ is the ideal sheaf of $\overline Y \subseteq V$ and $\beta':=\trace \circ \differential \rho-\beta \colon \mathfrak g \to \C$.

\begin{nota}\label{nota:BetaPrime}
  For a fixed representation $\rho \colon G \to V$, we use throughout the notation
  \[\beta' := \trace \circ \differential \rho - \beta\]
  for any Lie algebra homomorphism $\beta \colon \mathfrak g \to \C$.
\end{nota}

This notation originates from the fact that the operator $Z_V(\xi)-\beta'(\xi)$ is the negated Fourier transform on $V$ of the operator $Z_V(\xi)-\beta(\xi)$; those operators show up in the definition of tautological systems $\tau(\rho,\overline Y, \beta) = \FL(\hat\tau(\rho, \overline Y, \beta))$. Notice that because of $\FL \circ \mathbb D_{\mathscr D_V} = -\mathbb D_{\mathscr D_V} \circ \FL$, studying the dual of a tautological system or of its Fourier-transform are equivalent problems, and we choose to wwork with the Fourier-transformed tautological system $\hat\tau(\rho, \overline Y, \beta)$.

Note that, using the $G$-action on $V$, we have
\begin{align*}
  \hat\tau(\rho, \overline Y, \beta)
  &= \mathscr D_V \otimes_{\cA_V} \mathcal A_V/(\mathcal A_V \mathcal I_{\overline Y} + \mathcal A_V (\xi-\beta'(\xi) \mid \xi \in \mathfrak g')) \\
  &\cong \mathscr D_V \otimes_{\cA_V} ((\mathcal A_V/\mathcal A_V \mathcal I_{\overline Y})\{\beta'\} \otimes_{\cA_V} \cO_V)\\
  &\cong \mathscr D_V \otimes_{\cA_V} ((\mathcal A_V \otimes_{\O_V} \O_{\overline Y}\{\beta'\})\otimes_{\cA_V} \cO_V)  \\
  &\cong \mathscr D_V \otimes_{\cA_V} \O_{\overline Y}\{\beta'\},
\end{align*}
where the last isomorphism is due to \cref{lem:HTTanalogue}.

\begin{dfn}
  For $\rho \colon G \to V$, $\overline Y \subseteq V$ and $\beta \colon \mathfrak g \to \C$ as above and using \cref{nota:BetaPrime}, we define the derived Fourier-transformed tautological system
  \[\hat T(\rho, \overline Y, \beta) :=
  \cC^\bullet(\O_{\overline Y},\beta')
  = \mathcal S^\bullet_{\mathscr D_V|\mathcal A_V}(\mathscr D_V \otimes_{\O_V} \cO_{\overline Y}\{\beta'\}) \in D_{qc}^b(\mathscr D_V),\]
    where the definition of $\mathcal C^\bullet$ (and $\cA_V$) implicitly relies on the $G$-action on $V$ given by $\rho$.
      \end{dfn}

By \cref{lem:ComplexSRepresents},
in the derived category $D_{qc}^b(\mathscr D_V)$ of left $\mathscr D_V$-modules, we have
\[
\hat T(\rho, \overline Y, \beta) \cong (\mathscr D_V \otimes_{\cO_V} \cO_{\overline Y}\{\beta'\}) \otimes_{\cA_V}^\mathbb{L} \O_V \cong \mathscr D_V \otimes_\cA^\mathbb{L} \O_{\overline Y}\{\beta'\}
\]
where the last isomorphism is due to \cref{lem:HTTanalogue}.
In particular:
\[\hat\tau(\rho,\overline Y, \beta) = H^0 \hat T(\rho, \overline Y, \beta),\]
which is also easily checked from the definition of $\mathcal S_{\cdot|\cdot}^\bullet(\cdot)$ itself.

From \cref{cor:holonomy} and \cref{cor:CohomAmplitude}, we immediately obtain:
\begin{prop} \label{prop:THatHolonomic}
    Assume $\overline Y$ consists of finitely many $G$-orbits.
  Then $\hat T(\rho, \overline Y, \beta) \in D_h^b(\mathscr D_V)$, i.e., $\hat T(\rho, \overline Y, \beta)$ has holonomic cohomologies.
  (In particular,
    we obtain that
    $\hat{\tau}(\rho, \overline{Y},\beta)$ is holonomic, this is already known by  \cite{Hot98}). If $\overline Y$ is Cohen--Macaulay, then $H^i \hat T(\rho, \overline{Y}, \beta) = 0$ for $i \notin [\dim \overline Y-\dim G, \, 0]$.
\end{prop}
\begin{proof}
Only the last statement requires explanation. If $\overline{Y}$ is Cohen-Macaulay, then as stated in \cref{conv:OmegaM}, we have $\sheafExt_{\O_X}^i(\cM,\omega_X) = 0 \text{ for all }i \neq  \codim \cM
$. Therefore the integer
$k$ in \cref{cor:CohomAmplitude} equals $\textup{codim}_{V}(\overline{Y})$, and then $H^i \hat{T}(\rho,\overline{Y},\beta)=0$ for $i\notin [\dim V-\dim G -\textup{codim}_V(\overline{Y})=\dim\overline{Y}-\dim G,\, 0]$.
\end{proof}

\begin{lem}\label{lem:RelDual}
We have an isomorphism of left $\cA_V$-modules
$$
\omega_V^\vee \cong \cO_V\{\trace \circ \differential \rho\},
$$
where $\omega_V^\vee$ is a left $\cA_V$-module through its $G$-equivariant structure.
\end{lem}

\begin{proof}

  Through the choice of a basis of $V$, resp.\ the choice of a coordinate system, we get the explicit description $\omega_V^\vee = \bigwedge_{\cO_V}^{\dim V} \Theta_V = \O_V \partial_1 \wedge \dots \wedge \partial_N$ and the action of $\xi \in \mathfrak g$ on this equivariant line bundle is given by
  \begin{align*}
    \xi \cdot (\partial_1 \wedge \dots \wedge \partial_N) &= \sum_{i=1}^N \partial_1 \wedge \dots \wedge [Z_V(\xi),\partial_i] \wedge \dots \wedge \partial_N \\
    &= \sum_{i=1}^N  \partial_1 \wedge \dots \wedge [{\textstyle\sum_{j,k=1}^N -\differential\rho(\xi)_{kj} x_j \partial_k},\partial_i] \wedge \dots \wedge \partial_N \\
    &= \sum_{i=1}^N  \partial_1 \wedge \dots \wedge [-\differential\rho(\xi)_{ii} x_i \partial_i,\partial_i] \wedge \dots \wedge \partial_N \\
    &= \sum_{i=1}^N \differential\rho(\xi)_{ii} \partial_1 \wedge \dots \wedge \partial_N \\
    &= \trace (\differential \rho (\xi)) \ (\partial_1 \wedge \dots \wedge \partial_N),
  \end{align*}
  hence $\omega_V^\vee \cong \O_V\{\trace \circ \differential \rho\}$.
    \end{proof}

\begin{prop} \label{prop:DualityTautCM}
          Assume $\overline Y$ is Cohen--Macaulay and $n$-dimensional, let $m := \dim G$. Then
  \begin{align*}
  \mathbb{D}\,\hat T(\rho, \overline Y, \beta) &=
  \cC^\bullet(\omega_{\overline Y},2\trace \circ \differential\rho-\trace\circ \ad - \beta')[n-m] \\
  &= \cC^\bullet(\omega_{\overline Y},\trace \circ \differential\rho-\trace\circ \ad + \beta)[n-m].
  \end{align*}
    If $\overline Y$ consists of finitely many $G$-orbits, this implies
  \begin{align*}\mathbb{D} \, \hat \tau(\rho, \overline Y, \beta)   &\cong H^{n-m} \mathcal \cC^\bullet(\omega_{\overline Y}, \trace \circ \differential \rho-\trace \circ \ad + \beta).
  \end{align*}
\end{prop}

\begin{proof}
  The first claim follows from \cref{thm:DualityMain} together with \cref{lem:alphaDescription} and \cref{nota:BetaPrime}. Applying $H^0(\cdot)$ leads to the duality statement for $\hat\tau(\rho, \overline Y, \beta)$ because $H^0(\cdot)$ commutes with the duality functor in $\mathscr D_h^b(\mathscr D_V)$ (and we know by \cref{prop:THatHolonomic} that $\hat{T}(\rho,\overline{Y},\beta)$ has holonomic cohomologies under the assumption that $\overline Y$ consists of finitely many orbits).
\end{proof}

In the following, we will consider the more restrictive case where $\overline{Y}$ is supposed to be Gorenstein with trivial canonical bundle. More precisely, we assume that there is a Lie algebra homomorphism $\gamma \colon \mathfrak g \to \C$
  such that $\omega_{\overline Y} \cong \O_{\overline Y}\{-\gamma\}$ as left $\cA_V$-modules (where $\omega_{\overline{Y}}$ carries a left $\cA_V$-structure since it is an  $G$-equivariant $\cO_{\overline{Y}}$-module, see \cref{conv:OmegaM}).
\begin{thm} \label{thm:dualityTautGorenstein}
  Assume that $\overline Y$ is Gorenstein such that $\omega_{\overline Y} \cong \O_{\overline Y}\{-\gamma\}$ for some $\gamma \colon \fg \to \dC$. Then
  \[\mathbb{D}\, \hat T(\rho, \overline Y, \beta) =  \hat T(\rho, \overline Y, \tilde\beta)[n-m],\]
  where $\tilde\beta := \trace\circ \ad + \gamma - \beta$. Hence, if $\overline Y$ consists of finitely many $G$-orbits, then
  \begin{align*}\mathbb{D}\,\hat \tau(\rho, \overline Y, \beta) &= H^{n-m}\hat T(\rho, \overline Y, \tilde \beta).
    \end{align*}
\end{thm}

\begin{proof}
  From \cref{prop:DualityTautCM} and \cref{nota:BetaPrime}, we get:
  \begin{align*}
  \mathbb{D}\, \hat T(\rho, \overline Y, \beta)
  &= \cC^\bullet(\omega_{\overline Y},\trace \circ \differential\rho-\trace\circ \ad + \beta)[n-m] \\
  &= \cC^\bullet(\O_{\overline Y},\trace \circ \differential\rho-\trace\circ \ad - \gamma + \beta)[n-m] \\
  &= \hat{T}(\rho, \overline Y,\trace\circ \ad + \gamma - \beta)[n-m]. \qedhere
  \end{align*}
\end{proof}

We now consider the following situation:
For this, let $\rho_0 \colon G_0 \to \GL(V)$ be a representation of a group $G_0$. We extend this to a representation $\rho \colon G \to \GL(V)$ for $G := \C^* \times G_0$ by letting $\C^*$ act by simple scaling on $V$. Let $\overline Y \subseteq V$ be a $G$-orbit closure, which is then the affine cone over a projective $G_0$-variety $Z \subseteq \P V$.
In this setup, if $\overline Y$ is Gorenstein,
then its canonical bundle is trivial (since $\Pic(\overline Y) = 0$). Then by considering  a trivializing section of $\omega_{\overline{Y}}$,
it follows that $G$ must act via a group character $\chi:G\rightarrow \dC^*$, and then we have $\omega_{\overline Y} \cong \O_{\overline Y}\{-\gamma\}$, where $\gamma\colon \fg\to\dC$ is the derivative of $\chi^{-1}$.

\begin{exa}\label{ex:LCI-Gorenstein}
Suppose that $\overline{Y} \subseteq V$ is a complete intersection given as the vanishing $\{f_1 = \dots = f_k = 0\}$ of homogeneous polynomials $f_i$ of degree $d_i$. Then $\omega_{\overline Y} \cong \O_{\overline Y}\{-\gamma\}$ with $\gamma \colon \mathfrak g= \mathfrak g_0 \oplus \C\mathbf e \to \C$ given by $\gamma(\mathbf e) = \dim(V)-d_1-\dots-d_k
$ and, if $G_0$ is semisimple, $\restr{\gamma}{\mathfrak g_0} = 0$.

To see this, we observe that the Koszul complex on $\O_V$ given by $f_1,\dots,f_k$ is a $\dC^*$-equivariant resolution of $\O_{\overline{Y}}$ (at least when its terms are appropriately graded). Hence we have
$$
\omega_{ \overline{Y}} = \sheafExt_{\O_V}^k(\O_{\overline Y},\omega_V) \cong \cO_{\overline{Y}}\{d_1+\ldots+d_k - \dim(V)\}
$$
as $\dC^*$-equivariant $\cO_V$-modules, recall that $\omega_V\cong \cO_V\{-\trace\circ \differential\rho\}$ and that $\trace\circ \differential\rho(\mathbf e) = \dim V$.

On the other hand, we already know that $\overline{Y}$ that $\omega_{\overline{Y}}\cong \cO_{\overline{Y}}\{-\gamma\}$ for some $\gamma\colon \fg\to \C$. Therefore we must have $\gamma(\mathbf{e})=\dim(V)-d_1-\ldots-d_k$. Moreover, if $G_0$ is semi-simple, so that $[\fg_0,\fg_0]=\fg_0$, there are no non-trivial characters on $\mathfrak g_0$, hence $\restr{\gamma}{\mathfrak g_0} = 0$.
\end{exa}

\begin{exa}\label{ex:HomSpace}
  Consider the case that $\overline{Y}$ is the affine cone over a projective homogeneous space $G_0/P$ inside the embedding $G_0/P \hookrightarrow \P V$ by the complete linear system $|\Ell|$ of a very ample $G_0$-equivariant line bundle $\Ell$ on $X$. Assume $G_0$ is semi-simple and $\dim G_0/P > 0$.   The affine cone $\overline{Y}$ is a $G$-space for $G = \C^* \times G_0$, and it follows from \cite[Theorem 5]{Ramanathan} that it is Cohen-Macaulay, so that \cref{prop:DualityTautCM} applies.
  As shown in \cite[Theorem 5.1, Proposition 5.9 and Corollary 6.13]{GRSSW}, $\hat\tau(\rho, \overline Y, \beta)$ is non-zero only in the two cases $\beta = 0$ or $\beta(\mathbf e) = \ell/k$ with $\Ell^{\otimes \ell} \cong \omega_{G_0/P}^{\otimes (-k)}$ as $G$-equivariant line bundles.
  Consider the latter case with the assumption $\beta(\mathbf e) \in \Z$, i.e., $\Ell^{\otimes \ell} \cong \omega_{G_0/P}^\vee$ for some $\ell \in \Z$, then $\omega_{\overline Y \setminus 0} \cong \O_{\overline Y \setminus 0}\{-\beta\}$ (because we have $\cO_{G_0/P}(\ell)\cong \omega_{G_0/P}^\vee$ in this case).
  Since $\overline Y$ is normal and $\codim_Y\{0\} \geq 2$, and since $\omega_{\overline{Y}}$ is reflexive,   we have $\omega_{\overline Y} \cong \O_{\overline Y}\{-\beta\}$. In particular, $\overline Y$ is Gorenstein in this case. With \cref{thm:dualityTautGorenstein}, we conclude:
\begin{equation}\label{eq:DualHomSpaceGorenstein}
  \mathbb{D}_{\mathscr{D}_V}\, \hat\tau(\rho,\overline Y, \beta) = H^{n-m} \hat T(\rho, \overline Y, 0), \qquad n = \dim \overline Y, \ m = \dim G.
  \end{equation}
  When restricting to the open orbit in $\overline Y$, this recovers \cref{eq:DualityOrbit}.
  On the other hand, in the case $\beta(\mathbf e) = \ell/k \notin \Z$, from $\omega_{G_0/P}^{\otimes k} = \Ell^{\otimes (-\ell)}$ we similarly obtain $\omega_{\overline Y}^{\otimes k} \cong \O_{\overline Y}\{-k\beta\}$ as equivariant line bundles; hence, $\overline Y$ is $\Q$-Gorenstein (but not Gorenstein, since we know by \cite[Lemma 4.31 and Lemma 4.32]{GRSSW} that $\omega_{\overline{Y}\backslash 0}$ is not trivial in this case) and
  \[\mathbb{D}_{\mathscr{D}_V}\,\hat\tau(\rho,\overline Y, \beta) = H^{n-m} \cC^\bullet(\omega_{\overline Y},\trace \circ \differential\rho), \qquad n = \dim \overline Y, \ m = \dim G. \qedhere\]
\end{exa}

\section{Applications and Examples} \label{sec:App}

In the setup of the previous section, we can give more precise information on the dual of a tautological system under some assumptions on the dimension of the group $G$ compared to the dimension of the (closure of the) $G$-orbit $\overline{Y}$ used in the definition of the tautological system. We list below a few cases of interest.

\subsection{Case \texorpdfstring{$\dim(G)=\dim(\overline{Y})$}{dim(G)=dim(Y)}}

In this case, we obtain the most satisfying result as a direct consequence of \cref{thm:dualityTautGorenstein}:

\begin{cor}\label{cor:DualDimEqual}
Assume that $\overline{Y}\subseteq V$ is Gorenstein with
$\omega_{\overline Y} \cong \O_{\overline Y}\{-\gamma\}$ and such that $\dim(G)=\dim(\overline{Y})$. Assume further that $\overline Y$ consists of finitely many $G$-orbits. Then we have
$$
  \mathbb{D} \, \hat \tau(\rho, \overline Y, \beta) = \hat \tau(\rho, \overline Y, \trace\circ \ad + \gamma - \beta).
  $$
\end{cor}

We observe that the class of GKZ-systems given by a matrix of full row-rank fall into this case where the dimension of the group and the orbit agree. \cref{cor:DualDimEqual} generalizes the duality result for GKZ-systems \cite[Proposition~4.1]{Walther-Dual}.

\textbf{Application:} As a specific example of the construction described above, we consider a situation arising when studying so-called linear free divisors. We refer specifically to the setup of \cite[Section 4]{NarvaezSevenheck}. Namely, consider a reductive algebraic group acting prehomogeneously on a complex vector space $V_0$ (and denote by $x_1, \ldots, x_n$ a coordinate system relative to a chosen basis of $V_0$), meaning that there is a dense open orbit in $V_0$. Additionally, we require that the stabilizers for all points of this open orbit are finite groups. Let $D$ be the complement of the open orbit, and assume that $D\subseteq V$ is a divisor. The group we started with can then be characterized as $G_D:=\{g\in\GL(V)\,|\,gD=D\}$.
We also have $\dim(G_D)=\dim(V_0)=n$. Moreover, let $\xi_1,\ldots,\xi_n\in \fg_D$ be a basis of the Lie algebra of $G_D$, then let $A=(a_{ij})\in\textup{Mat}(n\times n, \dC[V]_1$) such that
$Z_V(\xi_j)=\sum_{i=1}^n a_{ij}\partial_{x_i}$ (where $Z_V$ is the notation for the anchor map used in \cref{ex:AModAsLieAlgebroid}). Then we require that
the determinant $f:=\det(A)\in \cO_V$ is reduced (it is then automatically a reduced equation of $D$, and necessarily we have that $f\in\dC[V]_n$). Under these conditions, we have that the module of vector fields $\Theta_V(-\log\,D)$ logarithmic along $D$ is free, and $D$ is called a linear free divisor (it is a free divisor in the sense of K.~Saito, see \cite{KSaito}, and it is called linear free since the polynomials $a_{ij}$ are linear forms on $V$). We will later further restrict to the class of so-called strongly Koszul (SK) free divisors, which in the current situation can simply be characterized by saying that the action of $G_D$ on $V$ has finitely many orbits.

We briefly recall some of the main constructions in \cite{NarvaezSevenheck}. First notice that there is a subgroup $G_0 \subset G_D$ (with $\dim(G_0)=n-1$) consisting of all linear transformations that stabilizes all fibres of $f$ (which was called $A_D$ in loc.cit.). Then we put $G:=\dC^*\times G_0$, $V=\dC\times V_0$ and we consider the extended action
$$
\begin{array}{rcl}
\rho: G & \longrightarrow & \Aut(V) \\ \\
(t,g) & \longmapsto & (v\mapsto t\cdot \rho_0(g)(v)),
\end{array}
$$
where $\rho_0$ is the restriction of the original prehomogeneous action to $G_0\subset G_D$. Let $Y_0:=f^{-1}(t)$ for $t\in\dC^*$, then $Y_0$ is a (closed) $G_0$-orbit. Let $Y:=C(Y_0)$ be its cone in $\dC^*\times V_0$, i.e. $Y=\rho(G)(1,p)$ for any point $p\in Y_0$.
In particular, we have an isomorphism $\iota:\dC^*\times Y_0\rightarrow Y$, $(t,y)\mapsto (t,t\cdot y)$.
Then $Y$ is a (usually non-closed) $G$-orbit, and we let $\overline{Y}\subset V$ be its closure. Clearly, the boundary $\partial Y$ is contained in $\{0\}\times V$. The main results of \cite{NarvaezSevenheck} then concerns the tautological system $\tau(\rho,\overline{Y},\beta)$. For simplicity, we will assume that $G_0$ is semi-simple, which then implies
that $\trace\circ \ad= 0$ and that $\beta_{|\fg_0}=0$, where $\fg=\dC\mathbf{e} \oplus \fg_0$. We cite the following result (notice that the various shifts and sign differences occuring here are due to the change of convention for the Euler field $Z_V(\mathbf{e})$ and for the definition of the module $\cO_{\dC^*}^{\beta(\mathbf{e})}$)
\begin{thm}[{\cite[Proposition 4.5]{NarvaezSevenheck}}]\label{thm:LFD-Taut}
Let $D\subset V$ be a SK linear free divisor defined by a group action $G_D\rightarrow \GL(V)$. Let $\rho:G:=\dC^*\times G_0 \rightarrow \GL(V)$ be as above, and denote by $k:Y\hookrightarrow V$ the locally closed embedding of the orbit $G$-orbit $Y$ into $V$. Then if
$\beta(\mathbf{e})\notin n\cdot (1+\roots(b_D))+\dZ_{> 0}$, we have
$$
\hat{\tau}(\rho, \overline{Y}, \beta) \cong (k\circ \iota)_+(\cO^{-\beta(\mathbf{e})}_{\dC^*}\boxtimes \cO_Y),
$$
where $\cO^{-\beta(\mathbf{e})}_{\dC^*}:=\cD_{\dC^*}/(t\partial_t-\beta(\mathbf{e}))$. In particular,
$\hat{\tau}(\rho, \overline{Y}, \beta)$ then underlies a complex Hodge module on $V$, i.e., an object of $\MHM(V,\dC)$,
which is an element in $\MHM(V)$, i.e. a rational Hodge module on $V$, if $\beta(\mathbf{e})\in \dZ$.
\end{thm}

From the basic functorial properties of the holonomic duality functor, we obtain the following consequence.
\begin{cor}\label{cor:LFD-Dual}
Under the hypotheses of the previous theorem, we have
$$
\bD_{\cD_V}
\hat{\tau}(\rho, \overline{Y}, \beta) \cong
(k\circ \iota)_\dag \left(\cO^{\beta(\mathbf{e})}_{\dC^ *}\boxtimes\cO_Y\right).
$$
If moreover $\beta(\mathbf{e})\in\frac12 \dZ$, then there is a morphism
$$
\bD_{\cD_V}
\hat{\tau}(\rho, \overline{Y}, \beta) \longrightarrow
\hat{\tau}(\rho, \overline{Y}, \beta).
$$
\end{cor}
\begin{proof}
If $\beta(\mathbf{e})\in\frac12 \dZ$, then $\cO^{\beta(\mathbf{e})}_{\dC^ *}\cong \cO^{-\beta(\mathbf{e})}_{\dC^ *}$, and the morphism is then simply given by the forgetful morphism from the properly supported to the ordinary direct image.
\end{proof}

Combining this last statement with our duality result \cref{cor:DualDimEqual} then yields the following.
\begin{prop}\label{prop:casesLFD}
Let $D\subset V$ be an SK-linear free divisor, and let $G$, $\rho$, $\overline{Y}$ be as above. Then
\begin{enumerate}
    \item
        If $\beta(\mathbf{e})\notin n\cdot (1+\roots(b_D))+\dZ_{\leq 0}$, then
        $$
            \hat{\tau}(\rho,\overline{Y},\beta)=
            (k\circ \iota)_\dag \left(\cO^{-\beta(\mathbf{e})}_{\dC^ *}\boxtimes\cO_Y\right).
        $$
        In particular, under this assumption, we also obtain that $\hat{\tau}(\rho,\overline{Y},\beta)$ underlies an object in $\MHM(V,\dC)$, and in $\MHM(V)$ if $\beta(\mathbf{e})\in \dZ$.
    \item
        If $\beta(\mathbf{e})\in\frac12\dZ \,\backslash \, (n\cdot (1+\roots(b_D))+\dZ_{> 0})$, then the natural duality morphism from \cref{cor:LFD-Dual} is expressed as
            $$
            \hat{\tau}(\rho, \overline{Y}, 1-\beta) \longrightarrow
            \hat{\tau}(\rho, \overline{Y}, \beta).
            $$

    \item If $\beta(\mathbf{e})\notin n\cdot (1+\roots(b_D))+\dZ$,  then $\hat{\tau}(\rho,\overline{Y},\beta)$ underlies a simple pure complex Hodge module on $V$.  The total Fourier-Laplace transform $\tau(\rho, \overline{Y}, \beta)$ (i.e. the actual tautological system) has therefore has irreducible monodromy representation.
\end{enumerate}
\end{prop}
\begin{proof}
\begin{enumerate}
    \item
        By \cref{cor:DualDimEqual}, and under our assumptions (using \cref{ex:LCI-Gorenstein}), we have
        \begin{equation}\label{eq:Dual-LFD}
          \mathbb{D}_{\mathscr D_V} \, \hat \tau(\rho, \overline Y, \beta) = \hat \tau(\rho, \overline Y, 1-\beta),
        \end{equation}
        where $1-\beta$ means the Lie algebra homomorphism $\mathfrak g \to \C$ trivial on $\mathfrak g_0$ and given on $\mathbf e$ by $1-\beta(\mathbf e)$.
        Then the statement follows from \cref{thm:LFD-Taut} by rewriting the condition on $1-\beta(\mathbf{e})$ as a condition on $\beta(\mathbf{e})$, and using the symmetry around $0$ of the set $1+\roots(b_D)$ shown in \cite{NarvaezDual}.
    \item
        Follows by combining the second statement in \cref{cor:LFD-Dual} with \cref{eq:Dual-LFD}.
    \item
        Under the assumptions made,
                we have that
        $$
            \hat{\tau}(\rho,\overline{Y},\beta)=
            (k\circ \iota)_{\dag\,+} \left(\cO^{-\beta(\mathbf{e})}_{\dC^ *}\boxtimes\cO_Y\right),
        $$
        and therefore it necessarily underlies a simple pure Hodge module. For the last statement, consider the total Fourier-Laplace transformation functor
        $$
            \FL:\textup{Mod}(\cD_V) \longrightarrow \textup{Mod}(\cD_{V^\vee}),
        $$
        which is well known to be an equivalence of categories. The tautological system $\tau(\rho,\overline{Y},\beta)$ is by definition the module $\FL(\hat\tau(\rho,\overline{Y},\beta))$, therefore, if $\hat\tau(\rho,\overline{Y},\beta)$ is a simple $\cD_V$-module, then
        $\tau(\rho,\overline{Y},\beta)$ is a simple $\cD_{V^\vee}$-module. The monodromy representation of the restriction of $\tau(\rho,\overline{Y},\beta)$ to its smooth part is then necessarily irreducible.
\end{enumerate}
\end{proof}

Note that the intersection of $n \cdot (1+\roots(b_D)) + \Z_{>0}$ and $n \cdot (1+\roots(b_D)) + \Z_{\leq 0}$ is a finite set. Therefore, the above in particular shows:

\begin{cor}\label{cor:AllButFinteLFD}
  Let $D \subseteq V$ be an SK-linear free divisor, and let $G$, $\rho$, $\overline Y$ be as above. Then, for all but at most finitely many values of $\beta(\mathbf e)$, the $\mathscr D_V$-module $\hat\tau(\rho, \overline Y, \beta)$ underlies an object in $\textup{MHM}(V,\C)$.
\end{cor}

In some special cases, where the roots of $b_D$ are known, we get an even sharper result. In particular, we know by \cite[Table 1]{Sev10}
that for the linear free divisors which
are discriminants in representation spaces for quivers of types $A_n$, $E_6$ and $D_m$ if $3 \nmid m-1$ that no two roots of $b_D$ differ by a multiple of $1/n$, so that the intersection of
 $n \cdot (1+\roots(b_D)) + \Z_{>0}$ and $n \cdot (1+\roots(b_D)) + \Z_{\leq 0}$ is necessarily empty, so that $\hat\tau(\rho, \overline Y, \beta)$ underlies an object in $\textup{MHM}(V,\dC)$ for all $\beta(\mathbf{e})$ in these cases.

\subsection{Case \texorpdfstring{$\dim(G)=\dim(\overline Y)+1$}{dim(G)=dim(Y)+1}} \label{sec:MIsNPlusOne}

We consider again the case of a representation $\rho_0 \colon G_0 \to V$ which we extend with the scaling action on $V$ to a representation $\rho \colon G \to V$ for $G = \C^* \times G_0$. Let $\overline Y$ be a $G$-orbit closure (this is in particular the affine cone over a projective $G_0$-variety $Z \subseteq \P V$). We then have the following general lemma (which does not depend on assumptions on $\dim(G)$ and $\dim(\overline{Y})$).

\begin{lem} \label{lem:bFct}
  Assume $\overline Y$ consists of finitely many $G_0$-orbits. For every Lie algebra homomorphism $\beta_0 \colon \mathfrak g_0 \to \C$ with $\hat\tau(\rho_0, \overline Y, \beta_0) \neq 0$, there exists a smallest monic univariate polynomial $b_{\beta_0} \neq 0$ such that $b_{\beta_0}(\dim V - Z_V(\mathbf e))$ lies in the ideal of $\hat\tau(\rho_0,\overline Y,\beta_0)$.

  If $\beta \colon \mathfrak g \to \C$ is such that $\restr{\beta}{\mathfrak g_0} = \beta_0$, then $\hat\tau(\rho,\overline Y,\beta) \neq 0$ if and only if $\beta(\mathbf e)$ is a root of $b_{\beta_0}$.
\end{lem}

We remark that we choose to consider $\dim V - Z_V(\mathbf e)$ for $\hat\tau(\rho_0,\overline Y, \beta_0)$ because, after Fourier-transform, this becomes the Euler vector field $E$ on $V^\vee$, meaning that $b_{\beta_0}(E)$ lies in the ideal of $\tau(\rho_0, \overline Y, \beta_0)$.

\begin{proof}
  Right-multiplication with $\dim V - Z_V(\mathbf e)$ defines a $\mathscr D_V$-linear endomorphism of the cyclic $\mathscr D_V$-module $\hat\tau(\rho_0,\overline Y, \beta_0)$ (it is well defined since the ideal of $\hat\tau(\rho_0,\overline Y, \beta_0)$ is
  $\dC^*$-equivariant).
  By \cref{cor:holonomy}, $\hat\tau_0 := \hat\tau(\rho_0,\overline Y, \beta_0)$ is holonomic, hence the vector space $\operatorname{Hom}_{\mathscr D_V}(\hat\tau_0,\hat\tau_0) = H^{\dim V} a_+ (\mathbb D \hat\tau_0 \otimes_{\O_V}^\mathbb{L} \hat\tau_0)$ is finite-dimensional. Therefore, the endomorphism given by right-multiplication with $\dim V - Z_V(\mathbf e)$ has a minimal polynomial, this is the polynomial $b_{\beta_0}$.

  For the second claim, consider $\beta \colon \mathfrak g \to \C$ with $\restr{\beta}{\mathfrak g_0} = \beta_0$ such that $\beta(\mathbf e)$ is not a root of the polynomial $b_{\beta_0}$. The ideal of $\hat\tau(\rho,\overline Y,\beta)$ is generated by $Z_V(\mathbf e)-\beta'(\mathbf e) = Z_V(\mathbf e) - \dim V + \beta(\mathbf e)$ and the ideal of $\hat\tau_0$, which contains $b(\dim V - Z_V(\mathbf e))$. Since the polynomials $b_{\beta_0}(s)$ and $s-\beta(\mathbf e)$ are assumed coprime, we have $p(s)b_{\beta_0}(s)+q(s)(s-\beta(\mathbf e))=1$ for some $p(s),q(s) \in \C[s]$. Plugging in $\dim V - Z_V(\mathbf e)$ into this polynomial equation, we conclude that $1$ lies in the ideal of $\hat\tau(\rho, \overline Y, \beta)$, hence $\hat\tau(\rho, \overline Y, \beta)=0$.
\end{proof}

\begin{exa}
If $\overline Y$ is the affine cone over a projective homogeneous space $G_0/P$ as in \cref{ex:HomSpace}, it turns out that there are only two $G_0$-orbits: $\{0\}$ and $\overline Y \setminus \{0\}$. Indeed, the way the parabolic subgroup acts on the one-dimensional linear subspace of $\overline Y$ spanned by $[1] \in G/P$, is given by a character $P\to \C^*$ and it is classically known that this data determines the equivariant line bundle $\Ell$ on $G_0/P$ uniquely. Since the very ample equivariant line bundle $\Ell$ is non-trivial, this means that $P$ acts non-trivially on this one-dimensional subspace of $\overline Y$. This implies that the $G$-orbit $\overline Y \setminus \{0\}$ is in fact already a $G_0$-orbit.

For $G_0$ semisimple, the are no non-trivial $\beta_0 \colon \mathfrak g_0 \to \C$. The $b$-function $b_0(s)$ is the polynomial
\[b_0(s) = \begin{cases}
s(s-\frac{\ell}{k}) &\text{if } \Ell^{\otimes \ell} \cong \omega_{G_0/P}^{\otimes (-k)} \text{ for some $\ell, k > 0$}\\
s &\text{otherwise.}
\end{cases}
\]
This follows from results in \cite{GRSSW}: The proof of Theorem~5.1 in loc.cit.\ shows that $b_0(s)$ divides $s(s-2\langle \delta , \mu \rangle/|\mu|^2)$, where
$\mu$ is the highest weight of the irreducible $G_0$-representation $V$ and $\delta$ is the half-sum of the positive roots of $G_0$. In the case $\Ell^{\otimes \ell} \cong \omega_{G_0/P}^{\otimes (-k)}$, the quantity $2\langle \delta, \mu\rangle/|\mu|^2$ equals $\ell/k$ by Proposition~5.9 in loc.cit. Finally, $b_0(s)$ is determined by the characterization that $\hat\tau(\rho, \overline Y, \beta) = 0$ except for exactly this case and the case of $\beta(\mathbf e) = 0$ (Corollary~6.13 in loc.cit.).
\end{exa}

\begin{prop}\label{prop:SelfDual}
  Consider the case $\dim G = \dim \overline Y + 1$. Assume that $\overline Y$ has finitely many $G_0$-orbits and is Gorenstein with  $\omega_{\overline Y} \cong \O_{\overline Y}\{-\gamma\}$.
        For every $\beta_0 \colon \mathfrak g_0 \to \C$, the $b$-function from \cref{lem:bFct}
        satisfies
  \[b_{\beta_0}(s) = b_{\delta_0+\gamma_0-\beta_0}(\gamma(\mathbf e)-s)\]
  for $\delta_0 := \restr{(\trace \circ \ad)}{\mathfrak g_0}$, $\gamma_0 := \restr{\gamma}{\mathfrak g_0}$.
  If $\beta \colon \mathfrak g \to \C$ is such that $\beta(\mathbf e)$ is a simple root of $b$, then
  \[\mathbb D \, \hat\tau(\rho, \overline Y, \beta) = \hat\tau(\rho, \overline Y, \trace \circ \ad + \gamma-\beta).\]
\end{prop}

\begin{rmk}
  If $G_0$ is semisimple, then $\mathfrak g_0 = [\mathfrak g_0,\mathfrak g_0]$, so there are no non-trivial characters $\beta_0 \colon \mathfrak g_0 \to \C$ and there is only $b := b_0$. Then we obtain the symmetry $b(s) = b(\gamma(\mathbf e)-s)$.
\end{rmk}

\begin{proof}
  Since $\dim G_0 = \dim \overline Y$, we know from \cref{cor:DualDimEqual} that
  \[\mathbb D\hat\tau(\rho_0, \overline Y, \beta_0) \cong \hat\tau(\rho_0, \overline Y, \delta_0 + \gamma_0 - \beta_0)\]
  Consider the endomorphism $\varphi$ of the cyclic $\mathscr D_V$-module $\hat\tau(\rho_0, \overline Y, \beta_0)$ given by right-multiplication with $\dim V - Z_V(\mathbf e)$. This dualizes to an endomorphism
  \[\hat\tau(\rho_0, \overline Y, \delta_0 + \gamma_0 - \beta_0) \cong \mathbb D\hat\tau(\rho_0, \overline Y, \beta_0) \xrightarrow{\mathbb D \varphi} \mathbb D\hat\tau(\rho_0, \overline Y, \beta_0) \cong \hat\tau(\rho_0, \overline Y, \delta_0 + \gamma_0 - \beta_0).\]
  By functoriality of the isomorphism in \cref{thm:DualityMain} and by tracking how this functoriality carries through the isomorphism in \cref{lem:RelDual} and $\omega_{\overline Y} \cong \O_{\overline Y}\{-\gamma\}$ to lead to \cref{thm:dualityTautGorenstein}, we see that $\mathbb D \varphi$ is given by right-multiplication with $\gamma(\mathbf e) - (\dim V - Z_V(\mathbf e))$. By functoriality of $\bD$, the morphisms $\varphi$ and $\mathbb D\varphi$ have the same minimal polynomial, therefore we can conclude that $b_{\beta_0}(s) = b_{\delta_0+\gamma_0-\beta_0}(\gamma(\mathbf e)-s)$.

  For the duality claim, note that $\hat\tau(\rho, \overline Y, \beta) = \coker(\varphi - \beta(\mathbf e)\id)$, since it arises from $\hat \tau(\rho_0, \overline Y, \beta_0)$ by quotienting out the left ideal generated by $-(Z_V(\mathbf e)-\beta'(e)) = (\dim V - Z_V(\mathbf e)) - \beta(\mathbf e)$. We have $\hat\tau(\rho_0,\overline Y, \beta_0) \cong \ker(\varphi-\beta(\mathbf e) \id) \oplus \im(\varphi-\beta(\mathbf e) \id)$, because $\beta(\mathbf e)$ is assumed a simple root of the minimal polynomial of $\varphi$. Hence, $\hat\tau(\rho, \overline Y, \beta) = \coker(\varphi - \beta(\mathbf e)\id) \cong \ker(\varphi - \beta(\mathbf e) \id)$. By dualizing, we see that $\mathbb D \hat \tau(\rho, \overline Y, \beta) \cong \coker(\mathbb D \varphi - \beta(\mathbf e) \id)$. As discussed before, the morphism $\mathbb D \varphi - \beta(\mathbf e) \id$ is right-multiplication with $\gamma(\mathbf e) -\beta(\mathbf e) - (\dim V - Z_V(\mathbf e))$ on $\hat \tau(\rho_0, \overline Y, \delta_0+\gamma_0-\beta_0)$. Its cokernel is by definition $\hat \tau(\rho, \overline Y, \delta+\gamma-\beta)$ for $\delta :=\trace \circ \ad$ (note that $\delta(\mathbf e) = 0$).
                       \end{proof}

\subsection{Homogeneous spaces}\label{sec:HomSpaces}

Consider again the case of the affine cone $\overline Y \subseteq V$ over a projective homogeneous space $X:=G_0/P$ as in \cref{ex:HomSpace}. Here, $V := \Gamma(X,\Ell)^\vee$ for a $G$-equivariant very ample line bundle on $X$ for $G = \C^* \times G_0$. Assume $G_0$ is semisimple, so that every Lie algebra homomorphism $\beta\colon \mathfrak g = \mathfrak g_0 \oplus \C \mathbf e \to \C$ is given by $\restr{\beta}{\mathfrak g_0} = 0$ and $\beta(\mathbf e) \in \C$.

Let $\beta(\mathbf e) = \ell \in \Z_{>0}$ and assume $\Ell^{\otimes \ell} = \omega_{G_0/P}^\vee$ (which is then the only case in which $\tau(\rho, \overline Y, \beta)$ is non-zero). In \cite[Theorem 6.3 and Corollary 6.13]{GRSSW}, it was shown that
\[\hat\tau(\rho, \overline Y, \beta) = H^0\iota_\dag \O_{\overline{Y} \setminus 0},\]
where $\iota \colon \overline Y \setminus 0 \hookrightarrow V$. Then
\[\mathbb D\hat\tau(\rho, \overline Y, \beta) = H^0\iota_+ \O_{\overline Y}\]
and by \cref{eq:DualHomSpaceGorenstein} in \cref{ex:HomSpace} this is $H^{n-m}\hat T(\rho, \overline Y, 0)$, where $n := \dim \overline Y$, $m:= \dim G$.

In particular, we have a natural morphism
\begin{equation} \label{eq:morphismToDual}
\hat\tau(\rho, \overline Y, \beta) = H^0 \iota_\dag \O_{\overline Y \setminus 0} \to H^0 \iota_+ \O_{\overline Y \setminus 0} = \mathbb D \hat\tau(\rho, \overline Y, \beta) = H^{n-m} \hat T(\rho, \overline Y, 0)
\end{equation}
whose image is $\iota_{\dag +}\cO_{\overline{Y}\backslash 0}$, i.e, the intersection cohomology $\mathscr D$-module of the singular variety $\overline{Y}$.

In order to given an explicit description of $\iota_{\dag +}\cO_{\overline{Y}\backslash 0}$ it would therefore be of interest to describe the morphism $\hat\tau(\rho, \overline Y, \beta) \to H^{n-m}\hat T(\rho, \overline Y, 0)$ explicitly. Recall that the latter is the cohomology of the complex $\mathscr D_V \otimes_{\O_V} \O_{\overline Y} \otimes_{\O_V} \bigwedge_{\O_V}^{-\bullet} (\O_V \otimes \mathfrak g)$ in cohomological degree $m-n = \dim G_0 - \dim X = \dim P$. It is natural to suspect that the parabolic subgroup $P$ and its Lie algebra $\mathfrak p$ lead through $\dim \bigwedge^{m-n} \mathfrak p = 1$ to a class of the $(n-m)$-th cohomology, to which the generator of the cyclic $\mathscr D_V$-module $\hat\tau(\rho, \overline Y, \beta)$ maps under \eqref{eq:morphismToDual}. More specifically, the locally free Lie algebroid $\mathcal E := \ker(\O_Y \otimes \mathfrak g \twoheadrightarrow \Theta_Y)$ of rank $m-n$ on $Y := \overline Y \setminus \{0\}$ has determinant $\bigwedge_{\O_Y}^{m-n} \mathcal E \cong \bigwedge_{\O_Y}^m (\O_Y \otimes \mathfrak g) \otimes_{\O_Y} \omega_Y^\vee \cong \O_Y$ by the Gorenstein property of $\overline Y$ as discussed in \cref{ex:HomSpace}.
This leads to a distinguished global section of $\bigwedge^{m-n}_{\O_Y} \mathcal E \subseteq \bigwedge^{m-n}_{\O_Y} (\O_Y \otimes \mathfrak g)$ and therefore of $\O_{\overline Y} \otimes \bigwedge^{m-n}_{\O_V} (\O_V \otimes \mathfrak g)$
which defines a class in $H^{n-m} \hat T(\rho, \overline Y, 0)$. We expect that the generator of $\hat \tau(\rho, \overline Y, \beta)$ is mapped to this global section by the morphism \eqref{eq:morphismToDual}. We postpone the details of this construction to a subsequent paper.

\begin{exa}
  For $G_0 = \SL(n)$ consider $G_0/P = \P^{n-1}$ and the equivariant line bundle $\Ell = \O_{\P^{n-1}}(d)$ with $d \mid n$ (so that $\beta(\mathbf e) = n/d \in \Z$).  $\overline Y$ is the cone over the $d$-th Veronese variety of $\P^{n-1}$. Note that $\mathfrak{sl}(n) \oplus \C \mathbf e \cong \mathfrak{gl}(n)$.
  The coordinates $x_1,\dots,x_n$ on $\C^n$ lead to coordinates on $V =\Sym^d \C^n$ being the degree~$d$ monomials in $x_1,\dots,x_n$.
    On $U := \overline Y \setminus \{x_n^d = 0\}$, the elements
  \[\{x_n^d \otimes E_{ij} - x_j x_n^{d-1} \otimes E_{in} \mid  i \leq n, \ j < n\} \subseteq \O_V \otimes \mathfrak{gl}(n)\]
  form a basis of $\ker(\O_U \otimes \mathfrak{gl}(n) \twoheadrightarrow \Theta_U)$. Their wedge product is
  \[\bigwedge_{i,j} (x_n^d \otimes E_{ij} - x_j x_n^{d-1} \otimes E_{in}) =
    \sum_{a \in \{1,\dots,n\}^n} (-1)^{\sum_{i} a_i} \, (x_{a_1}x_n^{d-1})(x_{a_2}x_n^{d-1})\dots (x_{a_n}x_n^{d-1}) \, \hat{E}_{1 a_1} \wedge \dots \wedge \hat{E}_{n a_n}.\]
  Here, $\hat{E}_{1 a_1} \wedge \dots \wedge \hat{E}_{n a_n}$ means the $(n^2-n)$-term wedge product obtained from $E_{11} \wedge E_{12} \wedge E_{13} \dots \wedge E_{nn}$ by removing the terms $E_{i a_i}$.
  The element $x_{a_1}\dots x_{a_n} \in \Sym^n (\C^n)^\vee$, which can be expressed as a product of $n/d$ linear forms on $V$ in several ways, but is a well-defined element of $\O_{\overline Y}$. Hence, in $\O_{\overline Y} \otimes_{\O_V} \bigwedge_{\O_V}^{n^2-n} \mathfrak{gl}(n)$, we may consider the element where the above element is divided by $x_n^{n(d-1)}$ (which was invertible on $U$):
  \[\zeta = \sum_{a \in \{1,\dots,n\}^n} (-1)^{\sum_{i} a_i} \, x_{a_1}x_{a_2}\dots x_{a_n} \, \hat{E}_{1 a_1} \wedge \dots \wedge \hat{E}_{n a_n} \in \O_{\overline Y} \otimes_{\O_V} \bigwedge_{\O_V}^{n^2-n} (\O_V \otimes \mathfrak{gl}(n)).\]
  One can check that $1 \otimes \zeta \in \mathscr D_V \otimes_{\O_V} \O_{\overline Y} \otimes_{\O_V}\bigwedge^{m-n}_{\O_V} (\O_V \otimes \mathfrak{gl}(n))$ lies in the kernel of the differential of the complex $\cC^{n-m}(\O_{\overline Y}, 0')$. We expect the morphism $\hat\tau(\rho, \overline Y, \beta) \to \D \hat\tau(\rho, \overline Y, \beta) = H^{n-m} \hat T(\rho, \overline Y, 0)$ to be given (up to a constant) by mapping the generator of $\hat\tau(\rho, \overline Y, \beta)$ to $1 \otimes \zeta$. In this example, it can be checked using some topological arguments that $\hat\tau(\rho, \overline Y, \beta)$ is self-dual, so this morphism will also be an isomorphism. Despite this speciality, the description of $\hat\tau(\rho, \overline Y, \beta) \to H^{n-m}\hat T(\rho, \overline Y, 0)$ in this example should be illustrative of the general case.
\end{exa}

\bibliographystyle{amsalpha-modified}

\providecommand{\bysame}{\leavevmode\hbox to3em{\hrulefill}\thinspace}
\providecommand{\MR}{\relax\ifhmode\unskip\space\fi MR }
\providecommand{\MRhref}[2]{  \href{http://www.ams.org/mathscinet-getitem?mr=#1}{#2}
}
\providecommand{\href}[2]{#2}

\vspace*{1cm}

\nd
Paul G\"orlach\\
Otto-von-Guericke-Universität Magdeburg\\
Fakult\"at f\"ur Mathematik\\
Institut f\"ur Algebra und Geometrie\\
Universitätsplatz 2\\
39106 Magdeburg\\
Germany\\
paul.goerlach@ovgu.de\\

\nd
Christian Sevenheck\\
Fakult\"at f\"ur Mathematik\\
Technische Universit\"at Chemnitz\\
09107 Chemnitz\\
Germany\\
christian.sevenheck@mathematik.tu-chemnitz.de\\

\end{document}